\documentclass[12pt]{article}
\usepackage{latexsym, amsmath, amsfonts, amsthm, amssymb}
\usepackage{times}
\usepackage{a4wide}
\usepackage{graphicx, tocloft}
\usepackage{mathrsfs}
\usepackage[dvipsnames]{xcolor}

\def \RR {\mathbb R}

\def \CC {\mathbb C}

\def \HH {\mathbb H}

\def \eps {\varepsilon}
\def \vphi {\varphi}

\def \cL {\mathcal L}

\def \rr {\mathbf{r}}

\newtheorem{theorem}{Theorem}[section]

\newtheorem{lemma}[theorem]{Lemma}

\newtheorem{proposition}[theorem]{Proposition}
\newtheorem{corollary}[theorem]{Corollary}

\theoremstyle{definition}
\newtheorem{remark}[theorem]{Remark}

\def\qed{\hfill $\vcenter{\hrule height .3mm
		\hbox {\vrule width .3mm height 2.1mm \kern 2mm \vrule width .3mm
			height 2.1mm} \hrule height .3mm}$ \bigskip}

\def \Re {{\rm  Re  }}

\def \susbeteq {\subseteq}

\begin{document}

\title{Horocyclic Brunn-Minkowski inequality}
\author{Rotem Assouline and Bo'az Klartag}
\date{}
\maketitle

\begin{abstract} Given two non-empty subsets $A$ and $B$ of the hyperbolic plane $\HH^2$, we define their horocyclic Minkowski sum
with parameter $\lambda=1/2$ as the set $[A:B]_{1/2} \subseteq \HH^2$ of all midpoints of horocycle curves connecting a point in $A$ with a point in $B$. These horocycle curves are parameterized by hyperbolic arclength. The horocyclic Minkowski sum with parameter $0 < \lambda <1$ is defined analogously. We prove that
when $A$ and $B$ are Borel-measurable,
$$ \sqrt{ Area( [A:B]_{\lambda} )} \geq (1-\lambda) \cdot \sqrt{ Area(A) } + \lambda \cdot \sqrt{ Area(B) }, $$
where $Area$ stands for hyperbolic area, with equality when $A$ and $B$ are concentric discs in the hyperbolic plane.
We also give horocyclic versions of the Pr\'ekopa-Leindler and Borell-Brascamp-Lieb inequalities.
These inequalities slightly deviate from the metric measure space paradigm on curvature and Brunn-Minkowski type inequalities, where the structure of a metric space is imposed on the manifold, and the relevant curves are necessarily  geodesics parameterized by arclength.
\end{abstract}

\section{Introduction}

The Brunn-Minkowski inequality is a geometric inequality
that was discovered circa 1887, see Schneider \cite[Section 7.1]{schneider} for its early history. In one of its formulations, it states that for any non-empty, Borel measurable subsets $A, B \subseteq \RR^n$
and $0 < \lambda < 1$,
\begin{equation} Vol_n( (1-\lambda) A + \lambda B )^{1/n} \geq (1-\lambda) \cdot Vol_n(A)^{1/n} + \lambda \cdot Vol_n(B)^{1/n}. \label{eq_1145}
\end{equation}
Here $Vol_n$ is $n$-dimensional volume in $\RR^n$, we write
$A + B = \{ x + y \, ; \, x \in A, y \in B \}$ for the Minkowski sum, and
$\lambda \cdot A = \{ \lambda x \, ; \, x \in A \}$. It is perhaps more common to formulate the Brunn-Minkowski inequality as stating that
\begin{equation} Vol_n( A + B )^{1/n} \geq  Vol_n(A)^{1/n} +  Vol_n(B)^{1/n}, \label{eq_1147}
\end{equation}
which is easily seen to be equivalent to (\ref{eq_1145}) by scaling. The advantage of the formulation (\ref{eq_1145}) is that the set $(1-\lambda) A + \lambda B$ is well-defined for any two sets $A$ and $B$ in an {\it affine} space, while the definition of the set $A + B$ requires the structure of a {\it linear} space, i.e., it requires a marked point in space referred to as the origin. The Brunn-Minkowski inequality is an indispensible tool in convex geometry, for instance it immediately implies the isoperimetric inequality, see
Burago and Zalgaller \cite{BZ},
Gardner \cite{gardner} or Schneider \cite{schneider}
for more information.

\medskip The Brunn-Minkowski inequality has been generalized to other geometries, and in particular to Riemannian manifolds. Given two subsets $A$ and $B$ of a complete, connected, $n$-dimensional, Riemannian manifold $M$, we may look at the set of all {\it midpoints} of minimizing geodesics connecting a point in $A$ with a point in $B$.
Observe that in the case where $M = \RR^n$, the set obtained is precisely $(A + B) / 2$. Similarly, we can make sense of $(1-\lambda) A + \lambda B $ for every $0<\lambda<1$. Provided that the Ricci curvature of $M$ is non-negative, it is known  that the Brunn-Minkowski inequality \eqref{eq_1145} holds true in $M$.
 See Cordero-Erausquin, McCann and Schmuckenschl\"ager \cite{CMS} for a proof of the Pr\'ekopa-Leindler inequality, which is a stronger, functional version of the Brunn-Minkowski inequality, in the Riemannian setting. For the case of constant curvature see also the earlier paper by Cordero-Erausquin \cite{DCE}. The Riemannian version of the Brunn-Minkowski inequality \eqref{eq_1145} in the case of non-negative Ricci curvature is discussed in Sturm \cite{sturm0, sturm}.

\medskip In negatively curved manifolds, the Brunn-Minkowski inequality may fail dramatically. Let us demonstrate this by looking at the example of the hyperbolic plane $\HH^2$. Consider two discs $A$ and $B$ of area one in $\HH^2$ that are far away from each other. It is not too difficult to show that the area of the set of midpoints
$$ \left \{ \gamma \left( \frac{1}{2} \right) \, ; \,
\gamma(0) \in A, \ \gamma(1) \in B, \textrm{and } \gamma:[0,1] \rightarrow \HH^2 \ \textrm{is a minimizing, constant speed geodesic }   \right \} $$
can be arbitrarily small. The existence of ``tiny bottlenecks'' through which all geodesics from $A$ to $B$ have to pass, roughly at their midpoint, is a fundamental property of negatively curved surfaces.

\medskip In this paper we propose to remedy the failure of the Brunn-Minkowski inequality in the hyperbolic plane by looking at horocycles instead of geodesics. The idea is that the geodesic curvature of the horocycles can, in some sense, compensate for the negative Gaussian curvature of the underlying hyperbolic space. This is loosely inspired by the success of the Bakry-\'Emery approach to Riemannian manifolds equipped with a measure, where
a negative Ricci tensor can be compensated by
the Hessian of the logarithm of the density.
See e.g. Bakry, Gentil and Ledoux \cite{BGL} for information on the Bakry-\'Emery theory.

\medskip
 A smooth curve $\gamma: [a,b] \rightarrow \HH^2$  is a {\it horocycle} if it has constant speed, and if its geodesic curvature equals $1$. For instance, consider the Poincar\'e disc model of hyperbolic geometry, in which the hyperbolic plane $\HH^2$ is represented in the unit disc $D = \{ z \in \CC \, ; \, |z| < 1 \}$ via the Riemannian metric tensor
 \begin{equation}  g =  \frac{4 |dz|^2}{(1- |z|^2)^2} \qquad \qquad (z \in D). \label{eq_1109} \end{equation}
In the Poincar\'e disc model, each horocycle is a Euclidean circle in the unit disc $D$ tangent to its boundary, i.e., a curve of the form
$$ \left \{ z \in D \, ; \, |z - a| = 1 - |a| \right \} $$
for some $0 \neq a \in D$. We emphasize that our horocycles are constant speed curves with respect to the hyperbolic metric,
since this is the only reasonable intrinsic way to
parametrize these curves. We refer the reader to Izumiya \cite{Iz} for a survey of horocyclical geometry,
in particular to the discussion of its resemblance to planar Euclidean geometry.
Horocycles can also be viewed as hyperbolic circles with center at infinity: the limit of a sequence of hyperbolic circles passing through a fixed point, whose center tends to infinity along some geodesic, is a horocycle.

\medskip
Any pair of points in the hyperbolic plane is joined by \emph{two} horocycles, and hence every two points have two different ``horocyclic midpoints". It is convenient to restrict our attention to \emph{oriented horocycles}: those horocycles whose velocity and acceleration vectors form an oriented orthogonal basis, after some choice of orientation has been made. In the Poincar\'e disc model, for example, we can restrict our attention to horocycles parametrized in the counterclockwise direction.
Thus, for any two distinct points $x,y \in \HH^2$ there is a unique oriented horocycle going from $x$ to $y$, which is different from the unique oriented horocycle going from $y$ to $x$.

\medskip For $A, B \subseteq \HH^2$ and for $0 < \lambda < 1$ we define the {\it $\lambda$-horocyclic Minkowski sum of $A$ and $B$}
as
\begin{equation}
	[A:B]_\lambda  = \left \{ \gamma(\lambda) \, ; \,  \gamma(0) \in A, \gamma(1) \in B, \ \textrm{and} \ \gamma \textrm{ is a constant-speed oriented horocycle} \right \}.
\label{eq_313} \end{equation}
When $A$ and $B$ are two concentric discs in $\HH^2$, the set $[A:B]_\lambda$ is again a disc. However, when the discs are not concentric, the $\lambda$-horocyclic Minkowski sum is not necessarily a disc, see Figure \ref{fig1}.

\begin{figure}
	\includegraphics[width = \textwidth]{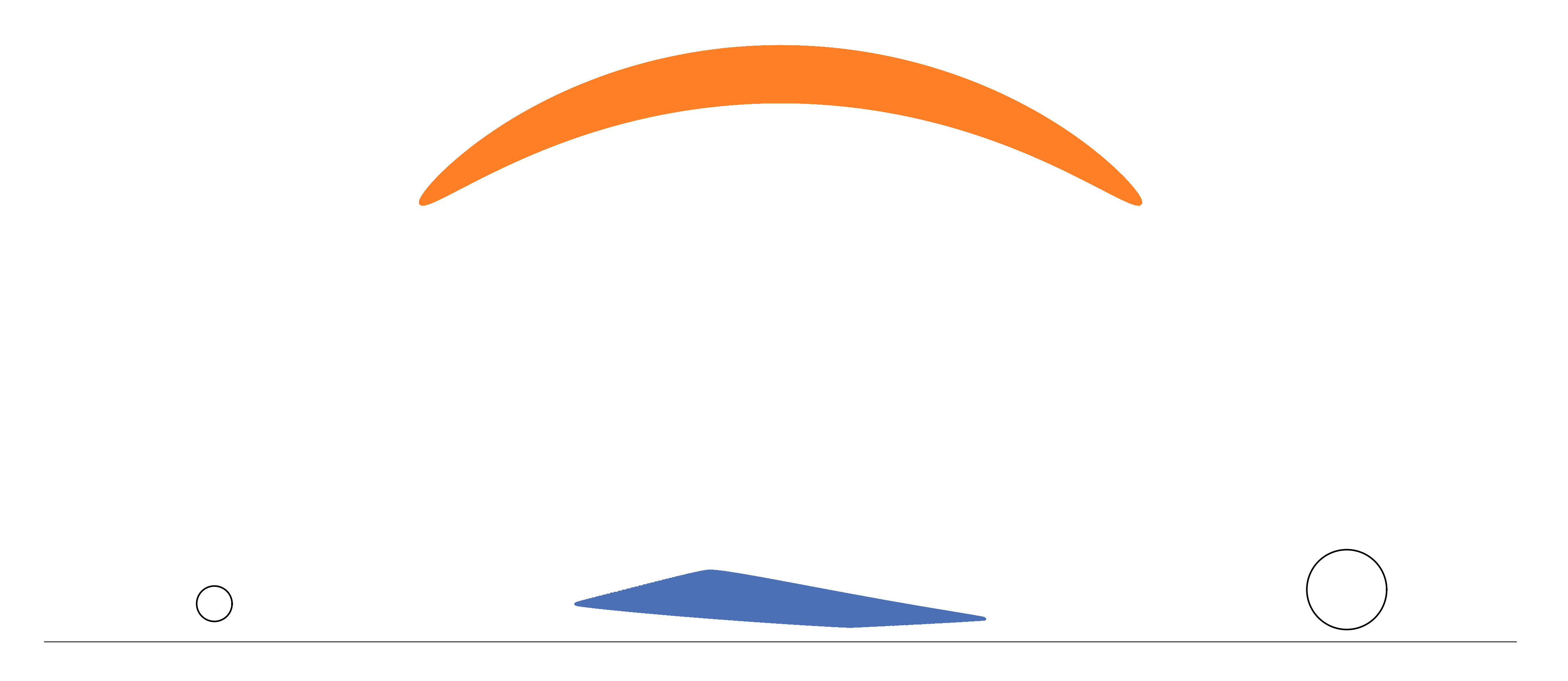}
	\caption{Horocyclic vs. geodesic Minkowski summation in the upper half plane model. The horocyclic $\frac12$-Minkowski sum of the two discs, which is the sum of midpoints of oriented horocycles joining the  disc on the left to the one on the right, is the {\color{Blue}{blue}} set in the middle. The  geodesic $\frac12$-Minkowski sum, i.e. the set of all midpoints of geodesics joining two discs, is the {\color{orange}{orange}} at the top. \label{fig1}}
\end{figure}

\begin{figure}
	\centering
	\includegraphics[width = .7\textwidth]{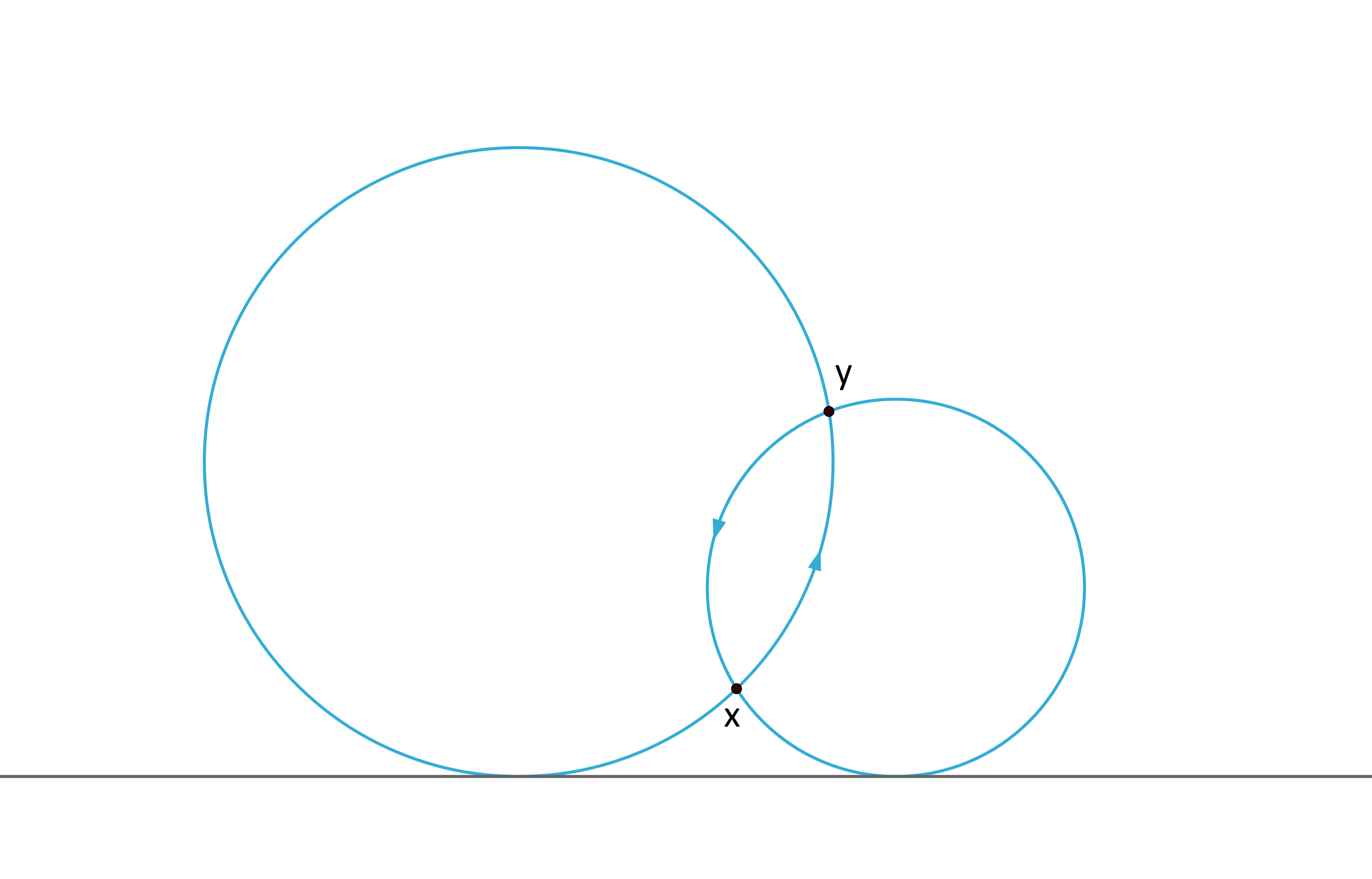}
	\caption{Horocycles joining two points $x,y \in \HH^2$, with the arrows indicating the orientation.}
\end{figure}

\medskip
Our main result is a Brunn-Minkowski theorem for the horocyclic Minkowski sum, which resembles the planar Euclidean Brunn-Minkowski inequality.

\begin{theorem} Let $A, B \subseteq \HH^2$ be non-empty, Borel measurable sets. Then for any $0 < \lambda < 1$,
	\begin{equation}  Area( [A:B]_\lambda )^{1/2} \geq (1-\lambda) \cdot Area(A)^{1/2} + \lambda \cdot Area(B)^{1/2} \label{eq_319} \end{equation}
where $Area$ stands for hyperbolic area in $\HH^2$. When $A$ and $B$ are concentric discs, or when one of the sets is a singleton, equality holds in (\ref{eq_319}).
\label{thm318}
\end{theorem}

\medskip Theorem 1.1 has a functional version in spirit of the Pr\'ekopa-Leindler and the Borell-Brascamp-Lieb inequalities, see Theorem \ref{horoBBL} below.

\begin{remark} Theorem \ref{thm318} remains correct if one
modifies definition (\ref{eq_313}) of the $\lambda$-horocyclic Minkowski sum, replacing ``an oriented horocycle'' by ``a horocycle''. When $A, B \subseteq \HH^2$ are concentric discs,
equality holds in (\ref{eq_319}) also with respect to this modified definition. \end{remark}

Theorem \ref{thm318} is analogous to the Brunn-Minkowski inequality in its formulation (\ref{eq_1145}) above. There is also a horocyclic analogue of \eqref{eq_1147}, see Remark \ref{horosuccinct}. We still do not know how to fully characterize the equality cases of (\ref{eq_319}). The higher dimensional case is briefly discussed in Remark \ref{r550} below. We remark that the operation \eqref{eq_313} is different from the horospherical Minkowski summations which appear in \cite{GST} and are defined using horospherical support functions.
The idea of using families of curves which are not necessarily geodesics of some metric in order to define Minkowski summation is explored further in a forthcoming paper \cite{A}. This research direction somewhat differs from the metric-measure space viewpoint on curvature and Brunn-Minkowski type inequalities, see Sturm \cite{sturm0, sturm} and Lott and Villani \cite{LV}, as well as from  the Lorentzian approach of McCann \cite{mccann}
and Cavalleti and Mondino \cite{CM2_}.

\medskip The proof of Theorem \ref{thm318} utilizes $L^1$-mass transport, or the {\it needle decomposition} approach for proving geometric inequalities on manifolds, that was suggested and developed in the Riemannian setting in \cite{K1}. This approach was extended to the
setting of a metric-measure space  by Cavalleti and Mondino \cite{CM1} and to the non-reversible Finsler setting by Ohta \cite{Oh}.

\medskip
The idea behind the needle decomposition technique is to localize the inequality by distintegrating the measure into one-dimensional measures called {\it needles}, and then to prove that the inequality holds on each individual needle. The first step is achieved in Theorem \ref{horodecomp}, and is a consequence of $L^1$-mass tansport theory. The second step is a consequence of a local property of horocycles which is given in Lemma \ref{detdFaffine}. A curious aspect of our proof is that we need to work with a certain auxiliary Finsler metric whose geodesics are {\it reparameterized} horocycles in order to construct the desired needle decomposition.
We learned about this Finsler metric from the work of Crampin and Mestdag \cite{CM2}.

\medskip The rest of this paper is organized as follows:
in Section \ref{backgroundsec} we recall general facts about Finsler metrics and horocycles. In Section \ref{proofsec} we state a theorem about horocyclic needle decomposition, Theorem \ref{horodecomp}, and use it to prove Theorem \ref{thm318}. In Section \ref{decompsec} we prove a needle decomposition theorem for Finsler manifolds, Theorem \ref{finslerdecomp}, from which Theorem \ref{horodecomp} follows. The results of subsections \ref{partsubsec} and \ref{finsdecompsubsec} are essentially known, but we believe that they might be useful to some readers as a relatively short exposition of the details behind the needle decomposition technique.

\medskip
{\it Acknowledgement.}  The second named author would like to thank Itai Benjamini for a fruitful discussion
that  led us to conjecture the horocyclic Brunn-Minkowski inequality and for his encouragement along the way.
We would like to thank the anonymous referee for reading the paper carefully and offering useful remarks. Supported by a grant from the Israel Science Foundation (ISF).

\section{Background}
\label{backgroundsec}

\subsection{Finsler structures}

Let $M$ be an $n$-dimensional smooth manifold. We write $T_x M$ for the tangent space at the point $x \in M$, and $TM = \cup_{x \in M} T_x M$ is the tangent bundle.

\medskip
A \emph{Finsler structure} (or \emph{Finsler metric}) on $M$ is a function $\Phi : TM \to [0,\infty)$, smooth away from the zero section, which satisfies the following
 requirements:

\begin{itemize}
	\item \emph{positively homogeneous}: $\Phi(\lambda v ) = \lambda \Phi(v)$ for all $v \in TM$ and $\lambda>0$.
	\item \emph{strongly convex}: Fix $x \in M$. Then the function $\Phi^2$ is convex in the linear space $T_x M$, and moreover  its Hessian at any point $0 \neq v \in T_x M$ is positive definite.
\end{itemize}

We refer the reader to Bao, Chern and Shen \cite{BCS} for more background on Finsler structures and for proofs of the basic facts mentioned in this section. A Finsler structure induces a metric on $M$ by setting
$$d(x,y) = d_M(x,y) = \inf_\gamma \mathrm{Length}(\gamma),$$
where
$$\mathrm{Length}(\gamma) : = \int_0^1\Phi(\dot\gamma(t)) dt,$$
and where the infimum is taken over all $C^1$ curves joining $x$ to $y$.
Since the function $\Phi$ is only positively homogeneous, $\Phi(-v)$ is not necessarily equal to $\Phi(v)$. As a result, the distance function may not be symmetric. In fact, the metric $d$ is symmetric if and only if $\Phi(v) = \Phi(-v)$ for all $v \in TM$, in which case $(M,\Phi)$ is said to be \emph{reversible}.

\medskip
Our notation does not fully distinguish between a parameterized curve $\gamma: [a,b] \rightarrow M$ and its image $\gamma([a,b])$ which is just a subset of $M$, sometimes endowed with an orientation. A curve has constant speed if $\Phi(\dot{\gamma}(t))$ is constant in $t$, and it has unit speed if $\Phi(\dot{\gamma}(t)) = 1$ for all $t$.

\medskip
 A \emph{minimizing geodesic} or a {\it forward-minimizing geodesic} is a constant-speed curve $\gamma:[a,b] \to M$ satisfying $\mathrm{Length}(\gamma) = d(\gamma(a),\gamma(b))$. A {\it geodesic} is a curve
$\gamma:[a,b] \to M$ that is locally a minimizing geodesic, i.e., for any $t_0 \in [a,b]$ there exists $\delta > 0$ such that the restriction of $\gamma$
to the interval $[a,b] \cap (t_0 - \delta, t_0 + \delta)$
is a minimizing geodesic. Equivalently, a geodesic is a solution to the Euler-Lagrange equation associated with the Lagrangian $\Phi^2/2$. We note that unless $\Phi$ is reversible, the curve $t \mapsto \gamma(-t)$ does not have to be a geodesic.

\medskip We recall that a minimizing geodesic $\eta$ cannot intersect a minimizing geodesic $\gamma:(a,b) \rightarrow M$ at more than one point unless they overlap.

\medskip
We say that $(M,\Phi)$ is \emph{geodesically convex} if any two points  $x,y \in M$ are joined by a minimizing geodesic.
We say that $(M,\Phi)$ is \emph{strongly convex} if any two points  $x,y \in M$ are joined by a \emph{unique} minimizing geodesic.
Any point in $M$ has a neighborhood that is strongly convex.
Any complete Finsler manifold is automatically geodesically convex.
We do not require below that $(M,\Phi)$ be complete.

\medskip Given $x\in M$ and a covector $\alpha \in T_x^*M$, the fiberwise strong convexity of $\Phi^2/2$ implies that the supremum
\begin{equation}\label{legendre}\sup_{v \in T_x M} \left[ \alpha(v) - \frac{\Phi^2(v) }{2} \right] \end{equation}
is uniquely attained at a tangent vector $\cL(\alpha) \in T_xM$, depending smoothly on $\alpha$. The resulting diffeomorphism $\cL : T^*M \to TM$ is called the \emph{Legendre transform} associated with the function $\Phi^2/2$. (In convexity theory it is  common to refer to the expression in (\ref{legendre})  as the Legendre transform, and not to the unique maximizer, but here we try to stick to  standard terminology
in Finsler geometry).

\medskip
Let $f : M \to \RR$ be a smooth function. The Legendre transform $\cL(df)$ of the differential of $f$ is referred to as the \emph{gradient} of $f$ and denoted by $\nabla f$. Note that the gradient $\nabla f(x)$ points at the direction of steepest infinitesimal ascent of the function $f$.

\medskip
One particular type of Finsler metrics, known as \emph{Randers metrics}, will be relevant to us. A Randers metric on a manifold $M$ is a Finsler metric of the form
	\begin{equation}\label{randersdefeq}\Phi = \sqrt{g} - \eta,\end{equation}
	where $g$ is a Riemannian metric on $M$ and $\eta$ is a 1-form on $M$. Thus $\eta$ is a fiberwise linear function on $TM$, while $g$ is a fiberwise quadratic form on $T M$, and consequently $\Phi$ is a well-defined
function on $TM$ that is fiberwise positively $1$-homogeneous. If $|\eta|_g < 1$ then $\Phi$ is a Finsler metric, see \cite{BCS} for more details. When $\dim M = 2$, Randers metrics have the following property, see Section 4 in \cite{CM2}.

\begin{proposition}\label{randersprop}
	Let $(M,g)$ be an oriented Riemannian surface, let $\eta$ be a 1-form on $M$ with $|\eta|_g < 1$, and let $\Phi = \sqrt{g} - \eta$ be the corresponding Randers metric. Write $d\eta = \kappa\, \omega_g$, where $\omega_g$ is the Riemannian area form. Then for every geodesic $\gamma$ of the Finsler metric $\Phi$, the signed geodesic curvature of $\gamma$ with respect to the Riemannian metric $g$ at a point $\gamma(t)$ equals $\kappa(\gamma(t))$.
\end{proposition}

{}\subsection{Horocycles in the hyperbolic plane}\label{horosubsec}

Here and throught the paper we shall use freely elementary facts about the hyperbolic plane, which can be found in standard textbooks, e.g \cite{Fen}. Fix an orientation of the hyperbolic plane $\HH^2$. By \emph{oriented horocycles} we mean constant-speed curves on $\HH$ of constant signed geodesic curvature 1. Thus a curve $\gamma$ is an oriented horocycle if $\nabla_{\dot\gamma}\dot\gamma = |\dot\gamma|\dot\gamma^{\perp}$, where $^\perp$ denotes rotation by $\pi/2$ in the oriented direction. Each pair of points $p,q\in\HH^2$ are joined by a unique oriented horocycle, whose length is $2 \sinh(d(p,q)/2)$, where $d$ is hyperbolic distance.

\medskip
Consider the Poincar\'e disc model of the hyperbolic plane, identifying $\HH^2$ with the unit disc of the complex plane $\{z \in \CC \, ; \, |z|<1 \}$ endowed with the Poincar\'e metric (\ref{eq_1109}).  The tangent space at any point of the unit disc will be naturally identified  with $\CC$. The collection of \emph{unit}-speed, oriented horocycles in the Poincar\'e disc is exactly the set of curves of the form
	\begin{equation}
	\alpha_{\lambda, t_0, \omega}(t) = \omega \cdot \frac{t - t_0 + (1-\lambda)i}{t - t_0 + (1+\lambda)i}
	\qquad \qquad \qquad (t \in \RR)
	\label{alphaformula}
	\end{equation}
	for $\lambda > 0$, $t_0 \in \RR$ and $\omega$ on the unit circle. Indeed, the Cayley transform $z \mapsto (z-i)/(z+i)$ maps the one-parameter family of unit-speed horocycles $\{t \mapsto  (i + t)/\lambda\}_{\lambda > 0}$ on the upper-half plane to the family $\{t\mapsto (t+(1-\lambda)i)/(t+(1+\lambda)i)\}_{\lambda > 0}$ of unit-speed oriented horocycles on the unit disc whose point at infinity is $1$. By applying an arbitary rotation and time-translation we get the formula \eqref{alphaformula}.
	

	\medskip
	Given a point $z$ in the unit disc and a unit tangent vector $v$ at $z$ of unit hyperbolic norm, the unique (unit-speed) oriented horocycle which at time $t$ visits $z$ with velocity  $v$, is represented by the parameters
	\begin{align}
	\begin{split}\label{parametereq}
	\lambda & = \frac{1 - |z| ^2}{| \hat v - i z |^2} \\
	t_0 & = t - \frac{2\Re(\hat v\bar z)}{| \hat v - i z |^2} \\
	\omega & = \frac{z + i \hat v}{1 + i \bar z \hat v} ,
	\end{split}
	\end{align}
	where $\hat v = 2v/(1-|z|^2)$ (this can be seen by a straightforward computation).
	In particular, $\omega, t_0, \lambda$ are smooth functions of $z, v$ and $t$.

\medskip It was observed in Crampin and Mestdag \cite{CM2} that there exists a Finsler structure on $\HH^2$ of Randers type whose geodesics are oriented horocycles, up to orientation-preserving {\it reparametrization}. In the Poincar\'e disc model, this Finsler structure is given by:
\begin{equation}
\Phi(x,y,u,v) = 2 \cdot \frac{ \sqrt{u^2 + v^2} + uy - xv}{1 - x^2 - y^2} \qquad (u,v) \in T_{(x,y)} D, \ (x,y) \in D
\label{horofinsler}
\end{equation}
where $D = \{ z \in \CC \, ; \, |z| < 1 \} \cong \{ (x,y) \in \RR^2 \, ; \, x^2 + y^2 < 1 \}$ and the tangent plane to $\RR^2$ at $(x,y)$ is identified with $\RR^2$ in the natural manner, with $(u,v) \in \RR^2$ being the coordinates in $T_{(x,y)} D$. Thus the metric $\Phi$ takes the form \eqref{randersdefeq} with $g$ being the Poincar\'e metric and with
$$\eta  = 2\cdot \frac{xdy - ydx}{1 - x^2 - y^2}, $$
which evidently satisfies $|\eta|_g < 1$ in $D$.
The fact that the geodesics of $\Phi$ are reparametrizations of oriented horocycles is a consequence of Proposition \ref{randersprop}: it is easy to check that
$$d\eta = 4\cdot\frac{dx\wedge dy}{(1-x^2-y^2)^2}$$
which is the hyperbolic area form, so the geodesics of $\Phi$ have signed geodesic curvature $1$ with respect to the hyperbolic metric.


\medskip
We shall need an extra fact about the Finsler structure $\Phi$ which is not mentioned in \cite{CM2}, namely, that all of its geodesics are length-minimizing.

\begin{lemma}\label{CMmin}
	Let $\Phi$ be defined as in \eqref{horofinsler}. Then for every $p,q \in \HH^2$ there is a unique geodesic of the Finsler metric $\Phi$  joining $p$ to $q$, and every geodesic of $\Phi$ is minimizing.
\end{lemma}

\begin{proof}
	Since geodesics of $\Phi$ are oriented horocycles, every $p,q \in \HH^2$ are joined by a unique geodesic of $\Phi$. It remains to prove that this geodesic is length-minimizing.

\medskip 
	Let $p,q \in \HH^2$, where now we view $\HH^2$ as the upper half plane with the Poincar{\'e} metric $g = (dx^2 + dy^2)/y^2$. 
In the upper half plane model, the Finsler metric $\Phi$ still takes the form $\Phi = \sqrt{g} - \eta$ with 
$d \eta = \omega_g$, where $\omega_g$ is the hyperbolic area form. 

\medskip Let $\gamma$ be an arbitrary $C^1$ curve joining $p$ to $q$ and let $\gamma_0$ denote the oriented horocycle joining $p$ to $q$. Let $T$ be an orientation-preserving hyperbolic isometry mapping the point $\tilde p = (0,1)$ to $p$ and the point $\tilde q = (r,1)$ to $q$, for some $r > 0$. Let $\tilde\gamma = T^{-1}\circ\gamma$ and $\tilde \gamma_0 : = T^{-1}\circ\gamma_0$. Since $T$ is an orientation-preserving isometry, $\tilde\gamma_0$ is the oriented horocycle joining $\tilde p$ to $\tilde q$, which is given by $\tilde\gamma_0(t) = (rt,1)$ for $0\le t \le 1$.
	Then
	$$\int_\gamma\Phi = \int_\gamma\left(\sqrt{g} - \eta\right) = \int_{\tilde \gamma}\left(\sqrt{T^*g} - T^*\eta\right) = \int_{\tilde\gamma}\left(\sqrt{g} - T^*\eta\right)$$
	since $T$ is an isometry. Moreover, since $d\eta = \omega_g$, we have $d(T^*\eta) = T^*(d\eta) = T^*\omega_g = \omega_g$ because $T$ is an isometry. Note that the 1-form $dx/y$ also satisfies $d(dx/y) = dx\wedge dy/y^2 = \omega_g$, so there exists a smooth function $f : \HH^2 \to \RR$ such that $T^*\eta = dx/y + df$. Thus
	\begin{align*}
		\int_\gamma\Phi & = \int_{\tilde\gamma}\left(\sqrt{g} - dx/y + df\right) \\
		& = \int_{\tilde\gamma}\left(\frac{\sqrt{dx^2 + dy^2}}{y} - \frac{dx}{y}\right) + f(\tilde q) - f(\tilde p )\\
		& \ge f(\tilde q) - f(\tilde p).
	\end{align*}
	On the other hand, by the same argument,
	$$\int_{\gamma_0}\Phi = \int_{\tilde\gamma_0}\left(\frac{\sqrt{dx^2 + dy^2}}{y} - \frac{dx}{y}\right) + f(\tilde q) - f(\tilde p) = f(\tilde q) - f(\tilde p)$$
	since along $\gamma_0$ we have $dy = 0$ and $dx > 0$. Thus $\gamma_0$ is the shortest curve from $p$ to $q$ with respect to $\Phi$. Moreover, equality implies that $dy = 0$ and $dx \ge 0$ along $\tilde\gamma$, whence $\tilde\gamma = \tilde\gamma_0$ (up to orientation-preserving reparametrization), and $\gamma$ is the oriented horocycle joining $p$ to $q$.
\end{proof}
The distance function induced by $\Phi$ has  the following simple description in the Poincar{\'e} disc model. We will not make use of this fact, but we find it interesting enough to mention.

\begin{proposition}
If $z,w$ are two points in the unit disc, then $d_{\Phi}(z,w) = \angle(z,p,w)$, where $p$ is the \emph{Euclidean} center of the oriented horocycle joining $z$ to $w$, and the angle is a Euclidean angle. In particular, the distance between any two points is less than $2\pi$.
\end{proposition}

\begin{proof}
		 Consider the parametrized circle $(1-r+r\cos t , r \sin t)$, $t \in [-\pi, \pi]$. Its Finslerian speed is 1 for any $0<r<1$. Indeed,
a direct calculation reveals that with $\Phi$ from (\ref{horofinsler}),
$$ \Phi( 1-r+r \cos t, r \sin t, -r \sin t, r \cos t) = 1. $$
Since this circle  is tangent to the boundary of the unit disc at the point $(1,0)$, it follows that it is a unit-speed geodesic with respect to the metric $\Phi$. By symmetry and homogeneity it follows that the Finslerian arclength along any horocycle is the same as \emph{angular speed} with respect to the Euclidean center of the horocycle. Since each horocycle is a minimizing geodesic with respect to $\Phi$, the proposition follows.
\end{proof}

\begin{remark}
	One might ask whether it is possible to find a Finsler structure $\tilde{\Phi}$ whose geodesics are oriented horocycles in their \emph{hyperbolic arc-length} parametrization. The answer is negative. Indeed, in such a metric the distance between any two points will be the length of the oriented horocycle joining them, and therefore the distance function will be given by
\begin{equation} d_{\tilde{\Phi}}(z,w) = 2\sinh(d_{g}(z,w)/2), \qquad z,w \in \HH^2, \label{eq_1010} \end{equation}
where $g$ denotes the hyperbolic metric. Let $v \in TM$ be a unit vector with respect to the hyperbolic metric and let $\gamma$ be a hyperbolic geodesic with initial velocity $v$. Then
$$\tilde\Phi(v) = \lim_{t \searrow 0}d_{\tilde\Phi}(\gamma(0),\gamma(t))/t = \frac{d}{dt}\Big\vert_{t = 0}2\sinh(t/2) = 1.$$
It follows that $\tilde \Phi$ coincides with the hyperbolic metric, a contradiction to (\ref{eq_1010}).
\end{remark}

%

\section{Proof of the inequality via horocyclic needle decomposition}
\label{proofsec}

The main tool we use in the proof of Theorem \ref{thm318} is the following theorem which, for each choice of two densities $\rho_1$ and $\rho_2$ with the same total mass, provides a disintegration of the hyperbolic measure into one dimensional measures (``needles''), each supported on a horocycle arc, such that the integrals of $\rho_1$ on $\rho_2$ on each needle are equal. The theorem will be used to localize the horocyclic Brunn-Minkowski inequality into a one-dimensional inequality on each needle.

\medskip
Throughout this section we shall denote by $\mu$  the area measure on the hyperbolic plane $\HH^2$.

\begin{theorem}\label{horodecomp}
	Let $\rho_1, \rho_2: \HH^2 \rightarrow [0, \infty)$ be $\mu$-integrable, compactly-supported functions with
	$$ \int_{\HH^2} \rho_1 d \mu = \int_{\HH^2} \rho_2 d \mu. $$
	Then there is a collection $\Lambda$ of disjoint oriented horocycle arcs, a measure $\nu$ on $\Lambda$ and a family $\{\mu_\gamma\}_{\gamma \in \Lambda}$ of Borel measures on $\HH^2$ such that the following hold:
	\begin{enumerate}
		\item[(i)] For all $\gamma\in \Lambda $, the measure $\mu_{\gamma}$ is supported on $\gamma$.
		\item[(ii)] (``disintegration of measure'') For any measurable set $S \susbeteq \HH^2$,
		\begin{equation}\label{horodisinteqn}
		\mu(S) = \int_{\Lambda} \mu_{\gamma}(S) d \nu(\gamma).
		\end{equation}
		
		\item[(iii)] (``mass balance'') For $\nu$-almost any $\gamma \in \Lambda$,
		\begin{equation}\label{MBhorodecomp}	\int_{\HH^2} \rho_1 d \mu_{\gamma} = \int_{\HH^2} \rho_2 d \mu_{\gamma}, \end{equation}
		and moreover
		\begin{equation}\label{MBhoroendsdecomp}	\int_{\HH^2} \rho_1 d \mu_{\gamma^+} \le \int_{\HH^2} \rho_2 d \mu_{\gamma^+} \end{equation}
		whenever $\gamma^+$ is a positive end of $\gamma$. Here a curve $\gamma^+$ is said to be a positive end of $\gamma$ if it is a restriction of $\gamma$ to a subinterval with the same upper endpoint, and the measure $\mu_{\gamma^+}$ is the restriction of $\mu_\gamma$ to the image of $\gamma^+$.
		\item[(iv)] (``affine density of needles'') For $\nu$-almost every $\gamma \in \Lambda$, the density of $\mu_\gamma$ with respect to arclength along $\gamma$ is affine-linear, except when $\gamma$ is a singleton, in which case the measure $\mu_\gamma$ is a Dirac mass.
	\end{enumerate}
\end{theorem}

Theorem \ref{horodecomp} is modeled after Theorem 1.2 in \cite{K1}, and its proof is very similar. In fact, since horocycles are geodesics of the Finsler metric \eqref{horofinsler}, conclusions (i)-(iii) of Theorem \ref{horodecomp} are essentially corollaries of Theorems 1.1 and 1.2 in \cite{Oh}, which extend \cite{K1} from Riemannian manifolds to Finsler manifolds. See also Theorem 2.7 in \cite{CM1}. The details of the proof of Theorem \ref{horodecomp} are discussed in the next section. For the rest of the present section we complete the proof of Theorem \ref{thm318} assuming Theorem \ref{horodecomp}.

\subsection{Proof of Theorem \ref{thm318}}
Recall the definition of the horocyclic Minkowski average: we fix an orientation of $\HH^2$ and define for $A,B \subseteq \HH^2$ and $0<\lambda<1$,
\begin{equation}\label{horosumdef}
	[A:B]_\lambda  := \left \{ \gamma(\lambda) \, ; \,  \gamma(0) \in A, \gamma(1) \in B, \ \textrm{and} \ \gamma \textrm{ is a constant-speed oriented horocycle} \right \}.
\end{equation}

\medskip
Let $A,B \subseteq \HH^2$ be nonempty and Borel measurable. The set $[A:B]_{\lambda}$ is the image of the Borel set $A \times B$ under
a continuous map, and it is well-known that such sets are Lebesgue measurable. We need to prove that
$$Area( [A:B]_\lambda )^{1/2} \geq (1-\lambda) \cdot Area(A)^{1/2} + \lambda \cdot Area(B)^{1/2}.$$
We may assume, by a standard approximation argument, that both sets are compact
(we work with a Radon measure, hence any Borel set contains a compact subset of approximately the same area). In particular,
both sets have finite area.

\medskip Consider first the case where both sets $A$ and $B$ have nonzero measure.
We apply Theorem \ref{horodecomp} with
\begin{equation}\label{rho12def} \rho_1 = \frac{\chi_A}{\mu(A)} \qquad \text{ and } \qquad \rho_2 = \frac{\chi_B}{\mu(B)}\end{equation}
and obtain measures $\{\mu_\gamma\}_{\gamma \in \Lambda}$ and $\nu$ with the properties (i)-(iv) above. Here $\chi_A$ is the indicator
function of the set $A$, which equals $1$ on $A$ and vanishes elsewhere.
The following lemma asserts that Brunn-Minkowski inequality holds on each individual needle $\mu_\gamma$.

\begin{lemma}\label{lem_924} For $\nu$-almost any $\gamma \in \Lambda$, if $0 < \mu_\gamma(A) < \infty$ then
\begin{equation}\label{muconcavity}
	\mu_\gamma([A:B]_\lambda)^{1/2} \ge (1-\lambda)\, \mu_\gamma(A)^{1/2} + \lambda\, \mu_\gamma(B)^{1/2}.
\end{equation}
\end{lemma}

We postpone the proof of Lemma \ref{lem_924} to subsection \ref{needlewisesubsec} below and proceed with the proof of Theorem \ref{thm318}. By \eqref{MBhorodecomp} and \eqref{rho12def}, for $\nu$-almost any $\gamma \in \Lambda$, if $0 < \mu_\gamma(A) <\infty$ then
\begin{equation}
\frac{\mu_{\gamma}(A)}{\mu(A)} = \frac{\mu_{\gamma}(B)}{\mu(B)}. \label{eq_1247}
\end{equation}
Since $A$ has finite measure, by \eqref{horodisinteqn} we know that $\mu_\gamma(A)  < \infty$ for $\nu$-almost any $\gamma \in \Lambda$. Thus by  \eqref{horodisinteqn}, \eqref{muconcavity} and \eqref{eq_1247},

\begin{align}
\begin{split}\label{finalcomputation}
	\mu\left([A:B]_\lambda\right) & \stackrel{\text{\eqref{horodisinteqn}}}{=} \int_{\Lambda}\mu_\gamma\left([A:B]_\lambda\right)d\nu(\gamma)\\
	& \stackrel{\text{\eqref{muconcavity}}}{\ge} \int_{\Lambda}\left((1-\lambda)\,\mu_\gamma(A)^{1/2} + \lambda\,\mu_\gamma(B)^{1/2}\right)^2d\nu(\gamma)\\
	& = \int_\Lambda\mu_\gamma(A)\left((1-\lambda) + \lambda\,\left(\frac{\mu_\gamma(B)}{\mu_\gamma(A)}\right)^{1/2}\right)^2d\nu(\gamma)\\
	& \stackrel{\text{\eqref{eq_1247}}}{=} \int_{\Lambda}\mu_\gamma(A)\left((1-\lambda) + \lambda\,\left(\frac{\mu(B)}{\mu(A)}\right)^{1/2}\right)^2d\nu(\gamma)\\
	& \stackrel{\text{\eqref{horodisinteqn}}}{=} \mu(A)\left((1-\lambda) + \lambda\,\left(\frac{\mu(B)}{\mu(A)}\right)^{1/2}\right)^2\\
	& = \left((1-\lambda) \, \mu(A)^{1/2} + \lambda\,\mu(B)^{1/2}\right)^2,
\end{split}
\end{align}
and inequality \eqref{eq_319} is proved.

\medskip
We now deal with the case where one of the sets has zero measure. In the case where $\mu(A) = \mu(B) = 0$, inequality (\ref{eq_319}) holds trivially.
Suppose that $\mu(A) = 0$ but $\mu(B) > 0$. Since $A$ is non-empty, we may pick a point $O \in A$. For $t > 0$ denote by
\begin{equation}\label{timesdef} t \times B := \left \{ \gamma(t) \, ; \, \gamma(0) = O, \gamma(1) \in B \, \textrm{and} \ \gamma\ \textrm{is a constant-speed oriented horocycle} \right \}, \end{equation}
the horocyclic dilatation of $B$ with respect to the point $O$.
 It is evident that
$$ \lambda \times B \subseteq [A:B]_{\lambda}. $$
Hence inequality (\ref{eq_319}) would follow once we show that for any $t > 0$,
\begin{equation}\label{horoscale}Area\left( t \times B \right) = t^2 \cdot Area(B), \qquad t >0.\end{equation}
To prove \eqref{horoscale}, consider ``horocyclic polar coordinates" centered at $O$, in which the point $(r,\theta)$ corresponds to $\gamma_\theta(r)$, where $\gamma_\theta$ is the unit-speed oriented horocycle emanating from $O$ at angle $\theta$ from some fixed direction. In the Poincar\'e disc model, if we take $O$ to be the origin, then it follows from (\ref{alphaformula})
that this coordinate map is $$ (r,\theta)\mapsto e^{i\theta} \cdot \frac{r}{r + 2i}. $$
By using formulae \eqref{Fformula} and \eqref{Jformula} with $\phi' = 1, t_0 = 0$ and $\lambda = 1$ (where $r,\theta$ are replaced by $t,y$),
it follows that in this coordinate system, the volume form is given by $r dr \wedge d\theta$, which implies \eqref{horoscale}. The case where $\mu(A) > 0$ and $\mu(B) =0$ can be reduced back to this case by reversing the orientation and changing the roles of $A$ and $B$.

\medskip
Inequality \eqref{eq_319} is therefore proven in all cases. We turn to the equality cases mentioned in Theorem \ref{thm318}. Suppose that $A$ is a singleton. Then $[A:B]_\lambda = \lambda \times B$, where in the definition of $\times$ we take the origin $O$ to be the unique member of $A$, and then by \eqref{horoscale}, $$ Area([A:B]_\lambda)^{1/2} = \lambda\cdot Area(B)^{1/2} = (1-\lambda)\cdot Area(A)^{1/2} + \lambda \cdot Area(B)^{1/2}. $$ If $B$ is a singleton, choose the other orientation of $\HH$ and apply the same argument to $[B:A]_\lambda$.

\medskip
Let us now analyze the case of concentric circles. Let $p \in \HH^2$ and $0<r_0<r_1$. Let $A,B$ denote the closed discs with center $p$ and radii $r_0,r_1$ respectively. The set $[A:B]_{\lambda}$ is connected, being the image of the connected set $A \times B$ under a continuous map.
By symmetry, $[A:B]_\lambda$ is a disc centered at the point $p$, and we denote its radius by $r_\lambda$. The area of a hyperbolic disc of radius $r$ is $4\pi \sinh^2(r/2)$, and therefore, in order to prove that equality holds in \eqref{eq_319}, we should prove that
\begin{equation}\label{sinhineq} \sinh(r_\lambda/2) =  (1-\lambda) \cdot \sinh(r_0/2)+ \lambda  \cdot \sinh(r_1/2).\end{equation}
By definition,
$$r_\lambda = \max_\gamma d(p,\gamma(\lambda)),$$
where the maximum is over all constant-speed oriented horocycles $\gamma$ with $\gamma(0) \in A$ and $\gamma(1) \in B$, i.e.,
\begin{equation}\label{gammaends} d(p,\gamma(0)) \leq r_0 \qquad \text{ and } \qquad d(p,\gamma(1)) \leq r_1. \end{equation}
 Thus to find $r_\lambda$ we need to maximize $d(p,\gamma(\lambda))$ over all constant-speed oriented horocycles $\gamma$ satisfying \eqref{gammaends}. If we let
$$h(p,q):=2\sinh(d(p,q)/2)$$
denote the length of a horocycle joining $p$ and $q$, then by monotonicity the same horocycle $\gamma$ would also maximize
$h(p,\gamma(\lambda))$.

\begin{lemma} The function $t\mapsto h(p,\gamma(t))$, for $\gamma$ a constant-speed horocycle, is convex.
\label{lem_904}
\end{lemma}

\begin{proof} By the symmetries of the hyperbolic plane, it suffices to consider the point $p=i$ in the upper half plane and the horocycle $\gamma(t) =  at + yi$ where $a,y>0$. In this case $$ h(i,at+yi) = \sqrt{a^2t^2 + (1-y)^2}/\sqrt y, $$
 which is a convex function of $t$.
 \end{proof}

By Lemma \ref{lem_904}, for every constant-speed horocycle $\gamma$ satisfying \eqref{gammaends},
\begin{align*}
	h(p,\gamma(\lambda)) & \le (1-\lambda) \cdot h(p,\gamma(0)) + \lambda \cdot h(p,\gamma(1)) \\
	& \leq (1-\lambda) \cdot 2 \sinh(r_0/2) + \lambda \cdot 2 \sinh(r_1/2).
\end{align*}
Equality is attained when $\gamma$ is a constant-speed horocycle through $p$
satisfying $\gamma(0) \in \partial A$ and $\gamma(1) \in \partial B$, because then
equality holds in \eqref{gammaends} and the function $t\mapsto h(p,\gamma(t))$ is affine-linear. Therefore,
$$ 2 \sinh (r_\lambda/2) = \max_\gamma h(p,\gamma(\lambda)) = (1-\lambda) \cdot 2 \sinh(r_0/2) + \lambda \cdot 2 \sinh(r_1/2),$$
which gives \eqref{sinhineq}. We have thus shown that equality is attained for concentric discs. This finishes the proof of Theorem \ref{thm318}.

\begin{remark}\label{horosuccinct}
The succinct formulation
(\ref{eq_1147}) admits a horocyclic analogue, if one is willing to make some non-canonical choices. Let $O \in \HH^2$ be a marked point viewed as ``the origin'', and for every $B \subseteq \HH^2$ and $t > 0$ define $t \times B$ as in (\ref{timesdef}).
Recall that the hyperbolic area of horocyclic dilatation of $B$ with respect to the point $O$ satisfies the scaling property (\ref{horoscale}).
If we denote $[A:B] = 2 \times [A:B]_{1/2}$, then Theorem \ref{thm318} implies that for any Borel measurable, non-empty sets $A, B \subseteq \HH^2$,
\begin{equation} Area([A:B])^{1/2} \geq Area(A)^{1/2} + Area(B)^{1/2} \label{eq_325} \end{equation}
with equality when $A$ and $B$ are concentric discs. Indeed, \eqref{eq_325} is a consequence of \eqref{eq_319} with $\lambda = 1/2$, together with (\ref{horoscale}).
\end{remark}

The proof of Theorem 1.1 can be modified  to produce the following horocyclic Borell-Brascamp-Lieb inequality, which contains \eqref{eq_319} as a special case. The case $p=0$ is a horocyclic version of the Pr{\'e}kopa-Leindler inequality. For $x,y \in \HH^2$ and $0<\lambda<1$ write $[x:y]_\lambda$ for the point $\gamma(\lambda)$, where $\gamma$ is the oriented constant-speed horocycle satisfying $\gamma(0) = x$ and $\gamma(1) = y$.

\begin{theorem} Let $f,g,h: \HH^2 \rightarrow [0, \infty)$ be three measurable functions, and assume that $f$ and $g$ are integrable with a non-zero integral.
Let $0 < \lambda <1$ and $p \in [-1/2, +\infty]$. Assume that for any $x,y \in \HH^2$ with $f(x) g(y) > 0$,
\begin{equation}  h \left( [x:y]_{\lambda} \right) \geq M_p(f(x), g(y) \, ; \lambda).
\label{eq_1719} \end{equation}
Define $q = p / (1 + 2 p)$, where by continuity $q = -\infty$ if $p = -1/2$ and $q = 1/2$ if $p = +\infty$. Then,
\begin{equation*}  \int_{\HH^2} h \geq M_{q} \left( \int_{\HH^2} f, \int_{\HH^2} g \, ; \lambda \right). \label{eq_1734} \end{equation*}
\label{horoBBL}
\end{theorem}

\begin{proof}
	We briefly sketch the proof, which follows the same lines as the proof of Theorem \ref{thm318}. Begin by applying Theorem \ref{horodecomp} to the functions
	$$\rho_1 : = \frac{f}{\int_{\HH^2} f} \qquad \text{ and } \qquad  \rho_2 : = \frac{g}{\int_{\HH^2} g}.$$
	(by a standard approximation argument one can assume that $f$ and $g$, and therefore also $\rho_1$ and $\rho_2$, are compactly supported). Next, we prove the analogue of Lemma \ref{lem_924}, namely, that for $\nu$-almost any $\gamma \in \Lambda$, if $0 < \int f d\mu_\gamma  < \infty$ then
	\begin{equation}\label{needleBBL}\int_{\HH^2}h d\mu_\gamma \ge M_q\left(\int_{\HH^2}fd\mu_\gamma,\int_{\HH^2}gd\mu_\gamma ; \lambda\right).\end{equation}
	In order to do so, fix some $\gamma \in \Lambda$, and pull back the needle measure $\mu_\gamma$ via the map $\gamma$ to obtain a measure on some interval $I \subseteq \RR$, which by Theorem \ref{horodecomp}(iv) has an affine-linear density $\ell$. Set
	$$F  = \ell \cdot (f \circ \gamma),  \qquad G  = \ell \cdot (g \circ \gamma), \qquad \text{ and } \qquad H  = \ell \cdot (h \circ \gamma).$$
	H{\"o}lder's inequality implies that with $\tilde{p} = p / (p+1)$,  for any $a_1,a_2, b_1, b_2 > 0$,
	$$ M_{p}(a_1,b_1;\lambda)\cdot M_{1}(a_2,b_2;\lambda) \ge M_{\tilde{p}}(a_1a_2,b_1b_2;\lambda). $$
	Here we interpret $\tilde{p}$ as $1$ when $p=\infty$, by continuity. Using this inequality, we conclude from the concavity of the function $\ell$ and from the condition \eqref{eq_1719} that
	$$H((1-\lambda)t + \lambda s) \ge M_{\tilde p}(F(t),G(s);\lambda)$$
	for any $t,s \in I$ with $t \le s$ and $F(t)G(s) > 0 $. We now apply Lemma \ref{dirBBL} to obtain \eqref{needleBBL}. The fact that the condition \eqref{sdeq} in Lemma \ref{dirBBL} holds is proved just as in the proof of Lemma \ref{sdlemma}. Finally, we use inequality \eqref{needleBBL} together with the disintegration of measure property \eqref{disinteqn} and the mass balance condition \eqref{MBhorodecomp} to finish the proof, as in the computation \eqref{finalcomputation}.
\end{proof}

\begin{remark} Consider the $n$-dimensional hyperbolic space $\HH^n$. Suppose that $A,B \subseteq \HH^n$ are two concentric balls, with radii $r_0, r_1$ respectively. Write $C \subseteq \HH^n$ for the collection of all midpoints of constant-speed curves of geodesic-curvature $1$ connecting a point in $A$ with a point in $B$.
	Then $C \subseteq \HH^n$ is a ball of radius $r_{1/2} > 0$,
	where relation (\ref{sinhineq}) holds true with $\lambda = 1/2$. This follows from an $n$-dimensional generalization of Lemma \ref{lem_904}. Hence,
	$$ f_n( Vol_n(C) ) = \frac{f_n(Vol_n(A)) + f_n(Vol_n(B))}{2}, $$
	where $Vol_n$ is hyperbolic volume, and $f_n(V) : = \sinh(R_n(V)/2)$ where $R_n(V)$ is the radius of a hyperbolic ball with volume $V$. 	
 In the two-dimensional case, $f_2(t) = c \sqrt{t}$ for some universal constant $c > 0$, but for $n \geq 3$ the function $f_n$ is not so simple, and a horocyclic Brunn-Minkowski inequality for $n \ge 3$, in any reasonable definition of the horocyclic Minkowski sum, will not resemble the Euclidean Brunn-Minkowski inequality as it does in dimension two.
	\label{r550}
\end{remark}

\subsection{Brunn-Minkowski for horocyclic needles}\label{needlewisesubsec}

In this subsection we prove Lemma \ref{lem_924}: for $\nu$-almost any $\gamma \in \Lambda$, if $0 < \mu_\gamma(A) < \infty$ then
 \begin{equation}\label{muconcavityagain}
	\mu_\gamma([A:B]_\lambda)^{1/2} \ge (1-\lambda)\, \mu_\gamma(A)^{1/2} + \lambda\, \mu_\gamma(B)^{1/2}.
\end{equation}

\begin{proof}[Proof of Lemma \ref{lem_924}]
Suppose first that $\gamma$ is a singleton, in which case the measure $\mu_\gamma$ is a Dirac mass at a point $x \in \HH^2$. By the mass-balance condition \eqref{MBhorodecomp}, either $x \in A \cap B$, in which case also $x \in [A:B]_\lambda$ and both sides of \eqref{muconcavityagain} are equal, or $x \notin A \cup B$, in which case the right hand side vanishes.

\medskip
Thus we may restrict our attention to the case where $\gamma: I \rightarrow \HH^2$ is an oriented, unit-speed horocycle arc, where $I \subseteq \RR$ is some interval, which is possibly a ray or the full line (but not a point), and the density  of $\mu_\gamma$ with respect to arclength is affine linear. Thus $\mu_\gamma$ is the pushforward via the map $\gamma$ of a measure $m$ on $I$ which has an affine density with respect to the Lebesgue measure on $I$. Set
$$\tilde A = \gamma^{-1}(A), \qquad \tilde B = \gamma^{-1}(B),$$
and write also
\begin{equation}\label{directedsumdef}[\tilde A : \tilde B]_\lambda : = \{(1-\lambda)t + \lambda s \mid t \in \tilde A, s \in \tilde B, t \le s\}.\end{equation}
By the definitions \eqref{horosumdef} and \eqref{directedsumdef},
 $$[\tilde A:\tilde B]_\lambda \subseteq \gamma^{-1}([A:B]_\lambda).$$
Indeed, if $t \in \tilde A$ and $s \in \tilde B$ and $t \le s$, then the curve $\tau \mapsto \gamma(t + (s-t)\tau))$ is a constant speed oriented horocycle which visits $A$ at time $0$ and $B$ at time $1$, so $\gamma((1-\lambda)t + \lambda s) \in [A:B]_\lambda$. Note that it is necessary to require $t \le s$ because we only include oriented horocycles in the definition \eqref{horosumdef}, so we cannot traverse $\gamma$ backwards. Hence in order to prove \eqref{muconcavityagain} it suffices to show that
\begin{equation}\label{mconcavity}m([\tilde A: \tilde B]_\lambda)^{1/2} \ge (1-\lambda) \, m(\tilde A)^{1/2} + \lambda\, m(\tilde B)^{1/2}.\end{equation}

\medskip
The density of the measure $m$ is affine, and in particular concave. The 1-dimensional Borell-Brascamp-Lieb inequality (see e.g. \cite[Section 10]{gardner}) then implies that
$$m((1-\lambda)\, \tilde A + \lambda\, \tilde B)^{1/2} \ge (1-\lambda)\, m(\tilde A)^{1/2} + \lambda\, m(\tilde B)^{1/2}.$$
But the set $[\tilde A : \tilde B]_\lambda$ can be smaller than $(1-\lambda)\, \tilde A + \lambda \, \tilde B$: the definition \eqref{directedsumdef} includes the condition $t \le s$. Thus we need a slightly modified version of the 1-dimensional Borell-Brascamp-Lieb inequality. It requires an extra condition on $\tilde A$ and $\tilde B$ which we now establish.

\begin{lemma}\label{sdlemma}
	For $\nu$-almost any $\gamma \in \Lambda$,
		\begin{equation}\label{meq}
			\frac{m(\tilde A \cap [t,\infty))}{m(\tilde A)} \le \frac{m(\tilde B \cap [t,\infty))}{m(\tilde B)} \qquad \text{ for every } t \in \RR.
		\end{equation}
\end{lemma}
\begin{proof}[Proof of Lemma \ref{sdlemma}]
	By the definition of $m, \tilde A$ and $\tilde B$, inequality \eqref{meq} is equivalent to
		$$\frac{\mu_\gamma(A\cap\gamma(I\cap[t,\infty)))}{\mu_\gamma(A)} \le \frac{\mu_\gamma(B \cap \gamma (I \cap [t,\infty)))}{\mu_\gamma(B)} \qquad \text{for every } t\in \RR.$$
		By \eqref{rho12def} and \eqref{eq_1247}, this is equivalent to
		$$\int_{\gamma(I\cap[t,\infty))}\rho_1d\mu_\gamma \le \int_{\gamma(I\cap[t,\infty))}\rho_2d\mu_\gamma \qquad \text {for every $t \in \RR$,}$$
		which is precisely the content of \eqref{MBhoroendsdecomp}.
\end{proof}
The following ``directed'' version of the one dimensional Borell-Brascamp-Lieb inequality is what we need in order to finish the proof of Lemma \ref{lem_924}. For $a,b \ge 0 $ and $p \in [-\infty,\infty]$ define
$$ M_p(a,b ; \lambda) = \left \{ \begin{array}{cc}
\big( (1-\lambda) a^p +  \lambda  b^p \big)^{1/p} & 0<p<\infty \text{ or }\\ &  -\infty< p < 0 \text{ and } ab > 0\\
a^{1-\lambda} b^{\lambda} & p = 0 \\
\max \{ a,b \} & p = +\infty \\
\min \{ a,b \} & p = -\infty
\\ 0 & -\infty< p < 0 \text{ and } ab = 0.
\end{array} \right.
$$
\begin{lemma}\label{dirBBL}
	Let $F,G$ be non-negative, integrable functions on an interval $I \subseteq \RR$ satisfying
	\begin{equation}\label{sdeq}
			\frac{\int_{I\cap[t,\infty)}F}{\int_IF} \le \frac{\int_{I\cap[t,\infty)}G}{\int_IG} \qquad \text{ for every } t \in I.
		\end{equation}
	Let $H$ be a non-negative integrable function satisfying
	\begin{equation}\label{dirBBLhyp}
		H((1-\lambda)t + \lambda s) \ge M_p(F(t),G(s);\lambda)
	\end{equation}
	for every $t,s \in I$ with $t \le s$ and $F(t)G(s) > 0$. Then
	\begin{equation}\label{dirBBLconc}
		\int_IH \ge M_{p/(p+1)}\left(\int_IF,\int_IG ; \lambda\right).
	\end{equation}
\end{lemma}
\begin{proof}
The proof follows the same lines as the usual transport proofs of the Borell-Brascamp-Lieb inequality in one dimension,
that go back at least to Henstock and Macbeath \cite{HM}. 	Define two functions $\mathbf{t},\mathbf{s}:[0,1] \to I$ by
		\begin{align*}
			\mathbf{t}(\xi) &:= \min\left\{t \in I \, ; \,  \int_{I\cap(-\infty,t]}F \ge \xi\int_IF\right\},\\
			\mathbf{s}(\xi) &:= \min\left\{t \in I \, ; \,  \int_{I\cap(-\infty,t]}G \ge \xi\int_IG\right\}.
		\end{align*}
		From assumption \eqref{sdeq} it follows that $\mathbf{t}(\xi) \le \mathbf {s}(\xi)$ for every $\xi \in [0,1]$.
The integral of $F$ on $I\cap(-\infty,t]$ is a continuous function of $t$, and similarly for $G$. Hence for any $\xi \in (0,1)$,
\begin{equation}  \int_{I \cap (-\infty, \mathbf{t}(\xi)]} F = \xi \int_I F \qquad \text{and} \qquad \int_{I \cap (-\infty, \mathbf{s}(\xi)]} G = \xi \int_I G.
\label{eq_1026_} \end{equation}
Thus the map $\mathbf{t}$ pushes forward the uniform measure on $[0,1]$ to the measure whose density is  $F / \int_I F$, and similarly
for $\mathbf{s}$ and the density $G / \int_I G$.
By the Lebesgue differentiation theorem, for almost any $\xi \in [0,1]$ the point $\mathbf{t}(\xi)$ is a Lebesgue density point of the function $F$,
and the point $\mathbf{s}(\xi)$ is a Lebesgue density point of the function $G$.
These two functions $\mathbf{t}(\xi), \mathbf{s}(\xi)$ are strictly increasing, and therefore differentiable almost
everywhere.  We may therefore differentiate (\ref{eq_1026_}), and obtain that for almost any $\xi \in [0,1]$,
$$ \mathbf{t}'(\xi) = \frac{\int_I F}{F(\mathbf{t}(\xi))} \neq 0 \qquad \text{and} \qquad \mathbf{s}'(\xi) = \frac{\int_I G}{G(\mathbf{s}(\xi))} \neq 0. $$
We now make the change of variables $\mathbf{r}(\xi) =  (1-\lambda)\mathbf{t}(\xi) + \lambda\mathbf{s}(\xi)$.
We claim that
\begin{equation} \int_0^1 H(\rr(\xi))\rr'(\xi)d\xi \le \int_IH.
\label{eq_1201}
\end{equation}
Indeed, let $\tilde{H}$ be a primitive of the integrable function $H$. The function $\xi \mapsto \tilde{H}(\mathbf{r}(\xi))$ is a non-decreasing
function, and hence the integral of its derivative on $[0,1]$ is at most
$$ \lim_{\xi \rightarrow 1^-} \tilde{H}(\mathbf{r}(\xi) - \lim_{\xi \rightarrow 0^+} \tilde{H}(\mathbf{r}(\xi)) = \int_I H. $$
(For the last fact, see e.g. \cite[Corollary 3.7 and the following Exercise 16]{Stein}).
Since the derivative of $\tilde{H}(\rr(\xi))$ equals $H(\rr(\xi))\rr'(\xi)$ almost everywhere in $[0,1]$, inequality (\ref{eq_1201}) is proven.

\medskip
The proof now proceeds by manipulating the integrand on the left-hand side of (\ref{eq_1201})
exactly as in the usual proof of the Borell-Brascamp-Lieb inequality, see Henstock and Macbeath \cite{HM} or Das Gupta \cite[page 300]{DasGupta},
while observing that these manipulations use (\ref{dirBBLhyp}) only for $t, s \in I$ with $t \le s$ and $F(t)G(s) > 0$.
\end{proof}

We can now finish the proof of Lemma \ref{lem_924}. Recall that the density of $m$ with respect to the Lebesgue measure is affine-linear, and denote it by $\ell$. We apply Lemma \ref{dirBBL} with
$$ F = \ell\cdot \chi_{\tilde A}, \qquad G = \ell\cdot \chi_{\tilde B}, \qquad H = \ell \cdot \chi_{[\tilde A : \tilde B]_\lambda} \qquad \text {and } \qquad p = 1$$
and obtain
\begin{equation}
	m([\tilde A : \tilde B]_\lambda) = \int_I H \geq M_{1/2} \left( \int_I F, \int_I G \, ; \lambda \right) = \left( (1-\lambda)\,m(\tilde A)^{1/2} + \lambda \, m(\tilde B)^{1/2} \right)^2\label{eq_2325}
\end{equation}
which is exactly \eqref{mconcavity}.
\end{proof}

\section{Horocyclic needle decomposition}
\label{decompsec}

In this section we prove the horocyclic needle decomposition Theorem \ref{horodecomp}. Needle decompositions on Riemannian manifolds are treated in \cite{K1}, and an extension to Finsler manifolds appears in \cite{Oh}. The purpose of the present section is to describe the construction of the needle decomposition in the two dimensional Finslerian case, since Theorem \ref{horodecomp} can be deduced from this case. As remarked already by Evans and Gangbo \cite{EG1}, the proof is somewhat simpler in dimension two, and we believe that including it here can be of value to some readers, as it offers a rather accessible roadmap to the proof of the needle decomposition results found in \cite{K1} and \cite{Oh}. We shall refer the reader to these sources for some parts of the proof. The main ideas behind the proofs in this section
go back to Evans and Gangbo \cite{EG1}, Feldman and McCann \cite{FM}, the second named author \cite{K1},
Cavalleti and Mondino \cite{CM1} and in particular Ohta \cite{Oh} who considered the non-reversible, Finslerian case.

\medskip A Borel measure on a smooth manifold $M$ is said to be  \emph{absolutely continuous} if it is absolutely continuous in  local coordinates. By the change of variables formula, the property of a Borel measure being {\it absolutely continuous with smooth density} does not depend on the choice of local coordinates. Similarly, we say that $A \subseteq M$ is a set of measure zero if in any local chart it has measure zero.
When we say that almost any point in $M$ satisfies a certain condition, we mean that the set of points that do not satisfy this condition is of  measure zero.

\medskip Suppose that $M$ and $N$ are smooth manifolds, $A \subseteq M$ an arbitrary set and $f: A \to N$. We say that $f$ is \emph{locally Lipschitz} if for every $p \in A$ there exist open neighbourhoods $p \in U \subseteq M$ and $f(p) \in V \subseteq N$, each contained in a coordinate chart, such that $f$ is Lipschitz in these local coordinates.

\subsection{Partition via a guiding function}\label{partsubsec}

\medskip
Let $(M,\Phi)$ be a geodesically-convex Finsler manifold. A function $u: M \rightarrow \RR$ is said to be $L$-Lipschitz if for all $x,y \in M$,
$$ -L \cdot d(y,x) \leq u(y) - u(x) \leq L \cdot d(x,y). $$
(note that the left-hand side inequality follows from the right-hand one).
The minimal $L \geq 0$ for which $u$ is $L$-Lipschitz
is denoted by $\| u \|_{Lip}$.

\medskip
Let $\mu_1,\mu_2$ be two absolutely-continuous, finite Borel measures on $M$ satisfying
	$$\mu_1(M) = \mu_2(M).$$
Consider the Monge-Kantorovich optimization problem
\begin{equation}\label{MK}
W_1(\mu_1, \mu_2) = \sup_{\|u \|_{Lip} \leq 1} \left[ \int_M u d \mu_2 - \int_M u d \mu_1 \right].
\end{equation}
It follows from the Arzela-Ascoli theorem that the supremum in (\ref{MK}) is finite and is actually a maximum. Let us fix a $1$-Lipschitz function $u$ that attains the supremum
in (\ref{MK}), and refer to this function as the {\it guiding function} or the {\it Kantorovich potential}. Define
$$ \Omega_u = \left \{ (x,y) \in M \times M \, ; \, u(y) - u(x) = d(x,y) \right \}. $$
The collection of \emph{strain points} of $u$ is
$$ Strain[u] = \left \{ x \in M \, ; \, \exists w,y \in M \setminus \{ x \}, \ \  (w,x), (x,y) \in \Omega_u \right \}, $$
and the collection of \emph{loose points} is
$$ Loose[u] = \left \{ x \in M \, ; \, \forall y  \in M \setminus \{ x \}, \ \ (x,y), (y,x) \not \in \Omega_u \right \}. $$
Clearly $Strain[u]$ and $Loose[u]$ are disjoint sets. A {\it transport ray} of $u$  is a minimizing, unit-speed geodesic $\gamma: I \rightarrow M$, with $I \subseteq \RR$ connected, open and non-empty, such that for all $s,t \in I$,
		\begin{equation}  u(\gamma(t)) - u(\gamma(s)) = t -s, \label{transportrayeqn} \end{equation}
and such that $\gamma$ is maximal: there is no minimizing, unit-speed geodesic $\tilde{\gamma}: J
 \rightarrow M$, with $J \subseteq \RR$ connected, open and strictly containing $I$, such that $\tilde{\gamma}|_I = \gamma$
and (\ref{transportrayeqn}) holds true for all $s,t \in J$. The restriction of a transport ray $\gamma : I \to M $ to a subinterval of the form $\tilde I = I \cap [t,\infty)$ for some $t \in \RR$, is  called a {\it positive end of a transport ray}.

\begin{proposition}[Properties of the guiding function]\label{uprops}
	\
	\begin{enumerate}
		\item[(i)] The function $u$ is differentiable at any point $x \in Strain[u]$, and its Finsler gradient $\nabla u(x) \in T_x M$ is a unit vector.
		
		\item[(ii)] The set $Strain[u]$ is the disjoint union of all transport rays of $u$. Furthermore,  each transport ray $\gamma: I \rightarrow M$ is an integral curve of $\nabla u$, i.e.
		\begin{equation}  \nabla u(\gamma(t)) = \dot{\gamma}(t) \qquad \qquad \text{ for all } t \in I. \label{eq_1855} \end{equation}
		\item[(iii)] If $(x,y) \in \Omega_u$ for $x \neq y$, then the relative interior of any forward-minimizing geodesic from $x$ to $y$ is contained in a transport ray.
		\item[(iv)] A union of transport rays, which is also a Borel subset of $M$, is called a \emph{transport set}. For any transport set $A \subseteq M$ we have the mass balance condition
		\begin{equation}\label{uMB} \mu_1(A) = \mu_2(A). \end{equation}
		A union of positive ends of transport rays, which is also a Borel subset of $M$, is called a \emph{positive end of a transport set}. If $A^+ \subseteq M$ is a positive end of a transport set then
		\begin{equation}\label{uMB2}\mu_1(A^+) \le \mu_2(A^+).\end{equation}
	\end{enumerate}
\end{proposition}

Proposition \ref{uprops} is well-known. We refer the reader to \cite[Section 4.1]{Oh} for the short proofs of (i)-(iii). The proof of the first part of (iv) appears in \cite[Section 5]{Oh}, while the second part is a direct consequence of this proof. In the Euclidean case a complete proof of (iv), which can be easily adapted to the current setting, can be found in \cite[Lemma 27]{CFM}.

\medskip
For the rest of this section we assume
$$\dim M = 2.$$
Our next step is to prove regularity properties of $\nabla u$ that will be used later
in the construction of the  disintegration of measure.
By Proposition \ref{uprops}(iii), every point $x \in Strain[u]$ is contained in a unique transport ray
\begin{equation}\label{rayeq} \gamma: (-\alpha(x),\beta(x)) \rightarrow M \end{equation} satisfying $$ \dot\gamma(0) = \nabla u (x). $$  By the definition of $Strain[u]$, both $\alpha$ and $\beta$ are well-defined, positive functions on $Strain[u]$, possibly attaining the value $+\infty$. These two functions are Borel mesurable, essentially by the same arument as in \cite[Lemma 2.1.12]{K1}. Let
$$Strain_\eps[u] : = \{ x \in Strain[u] \, ; \, \alpha(x),\beta(x) \ge \eps \}.$$
Thus
$$Strain[u] = \bigcup_{\eps>0}Strain_\eps[u].$$

\begin{lemma}\label{locliplemma}
	For any $\eps > 0$, the gradient $\nabla u$ is locally Lipschitz in $Strain_{\eps}[u]$.
\end{lemma}

The proof of Lemma \ref{locliplemma} requires the following two-dimensional lemma, whose proof is deferred to the Appendix below.


\begin{lemma}[``Disjoint geodesics that are close to each other at one point have similar tangents'']\label{tanglemma} Let $p \in M$. Then there exist $C,c > 0, \sigma_0 \in (0,1)$ and a smooth coordinate chart $\psi: V \to U$, where $V$ is an open subset of $\RR^2$ and $U$ is a neighbourhood of $p$, such that the following holds.
	
	\medskip  Suppose that
	$0 < \sigma \leq \sigma_0$ and $\gamma_1, \gamma_2: (-\sigma, \sigma) \rightarrow U$ are two disjoint, unit-speed geodesics. Set
$$
\eta_i = \psi^{-1} \circ \gamma_i, $$
	and assume that
	$|\dot{\eta}_1(0) - \dot{\eta}_2(0)| < c$. Then,
	$$  |\dot{\eta}_1(0) - \dot{\eta}_2(0)| \leq \frac{C}{\sigma} \cdot |\eta_1(0) - \eta_2(0)|. $$
\end{lemma}

\begin{proof}[Proof of Lemma \ref{locliplemma}]
	First, we argue that $\nabla u$ is continuous on $Strain_\eps[u]$. Let $p_n \in Strain_\eps[u]$ converge to $p \in Strain_\eps[u]$. Then there exist unit-speed, minimizing geodesics $\gamma_n,\gamma:(-\eps,\eps) \to M$ with
$$\gamma_n(0) = p_n, \quad \gamma(0) = p, \quad \dot \gamma_n (0) = \nabla u (p_n) \quad \text{ and } \quad \dot\gamma(0) = \nabla u(p),$$
all satisfying \eqref{transportrayeqn} for all $t,s\in(-\eps,\eps)$. Assume by contradiction that $\dot\gamma_n(0)$ has a subsequence $\dot\gamma_{n_k}(0)$ converging to some $v \ne \nabla u(p)$. By the continuity of $u$ and of the Finslerian distance function $d$, the pointwise limit of $\gamma_{n_k}$ is a unit-speed, minimizing geodesic $\gamma_0$ satisfying $\gamma_0(0) = p$, $\dot \gamma_0(0) = v$ and $$ u(\gamma_0(t)) - u(\gamma_0(s)) = t - s \qquad \qquad \text{for all} \ t,s \in (-\eps,\eps).$$ Since $u$ is 1-Lipschitz, it follows that $\nabla u(p) = v$, in contradiction. This completes the proof that $\nabla u$ is continuous. Next, let $p \in Strain_\eps[u]$. Let the coordinate chart $\psi : V \to U$ and the constants $C,c,\sigma_0 > 0$ be as in Lemma \ref{tanglemma}. For ease of reading and with abuse of notation, in this proof we identify $U \subseteq M$ with $V\subseteq\RR^2$, and we identify $TU$ with $TV \cong V \times \RR^2$.
Since $\nabla u$ is continuous on $Strain_{\eps}[u]$, we may shrink the coordinate chart and assume that
	\begin{equation}\label{graddifsmall}|\nabla u(q_1)- \nabla u(q_2)| < c \qquad \text{ for all } q_1, q_2 \in U \cap Strain_{\eps}[u].\end{equation}
	Choose a neighbourhood  $U' \subset U$ of $p$ and a number $\sigma_1 > 0$ such that any geodesic $\gamma:(-\sigma_1,\sigma_1) \to M$ with $\gamma(0) \in U'$ is contained in $U$.
Set $$ \sigma = \min\{\sigma_0, \sigma_1, \eps\}. $$
For any
	$q \in U'\cap Strain_\eps[u]$, there exists a unit-speed geodesic $\gamma_q:(-\sigma,\sigma) \to U$, contained in the transport ray through $q$, with $\gamma_q(0) = q$ and
$$ \dot\gamma_q(0) = \nabla u(q). $$ Let $q_1, q_2 \in U' \cap Strain_{\eps}[u]$. If $q_1,q_2$ belong to different transport rays, then $\gamma_{q_i}$ are disjoint by Proposition \ref{uprops}(iii), hence by \eqref{graddifsmall} and Lemma \ref{tanglemma},
	$$|\nabla u(q_1) - \nabla u (q_2)| \le \frac C \sigma |q_1 - q_2|.$$
	If $q_1,q_2$ lie on the same transport ray then $\nabla u(q_1)$ and $\nabla u(q_2)$ are unit vectors tangent to the same geodesic,
	and by the smoothness of geodesics, they differ by at most $C'|q_1 - q_2|$, where $C'$ is independent of $q_1,q_2$. It follows thus that $\nabla u$ is locally Lipschitz on $U' \cap Strain_{\eps}[u]$.
\end{proof}

The set $Strain[u]$, which is a disjoint union of transport rays, can be divided into countably many transport sets, each admitting a convenient parametrization.

\begin{lemma}\label{decomplemma}
	We may decompose $Strain[u]$ into a countable disjoint union of transport sets $A_1,A_2,\ldots$ such that for any $k \geq 1$,
	\begin{enumerate}
		\item[($\star$)]  There exists a Borel subset $Y_k  \subseteq \RR$ and a locally Lipschitz one-to-one map $\vphi_k: Y_k \rightarrow A_k$
		such that for each transport ray $\gamma \subseteq A_k$, the set $\vphi_k(Y_k) \cap \gamma$ is a singleton. Moreover,
		$\nabla u$ is locally Lipschitz in $\vphi_k(Y_k)$.
	\end{enumerate}
\end{lemma}

\begin{proof}
	Let $p \in Strain[u]$. Then $p \in Strain_\eps[u]$ for some $\eps > 0$. Let  $v \in T_pM$, $v \ne \nabla u(p)$ and let $\eta: (-s_0, s_0) \rightarrow M $ be a minimizing geodesic with $\dot{\eta}(0) = v$, for some $s_0 > 0$. Thus $\eta$  is transverse to the transport ray through $p$.
	Since a transport ray is a minimizing geodesic and since $\eta$ is a minimizing geodesic defined on an open interval, no transport ray can intersect $\eta$ more than once.

	\medskip
	Let $A_{p,\eps}$ denote the union of all transport rays intersecting $\eta \cap Strain_\eps[u]$. Then $A_{p,\eps}$ is Borel (even sigma-compact, as in \cite[Lemma 2.1.12]{K1}) hence is a transport set. We claim that there exists an open disc $D$ containing $p$ such that $D\cap Strain_\eps[u] \subseteq A_{p,\eps}$. Indeed, otherwise there exists a sequence $p_k \in Strain_\eps[u]$ converging to $p$ such that the transport rays through $p_k$ do not intersect $\eta$,
	but the transport ray through $p$ is transverse to $\eta$, which is a contradiction to the continuity of $\nabla u$ on $Strain_\eps[u]$ (since we are in dimension two).

	\medskip
	Let $Y_{p,\eps} = \eta^{-1}(Strain_\eps[u])$, and $\varphi_{p,\eps} = \eta\vert_{Y_{p,\eps}}$. Then $Y_{p,\eps}$ is Borel and $\varphi_{p,\eps}$ is locally Lipschitz and one-to-one, and $\nabla u$ is locally Lipschitz on $\varphi_{p,\eps}(Y_{p,\eps})$ by Lemma \ref{locliplemma}. By construction, $\varphi_{p,\eps}(Y_{p,\eps})\cap\gamma$ is a singleton for every transport ray $\gamma \subseteq A_{p,\eps}$.

	\medskip
	Since each $A_{p,\eps}$ contains a neighbourhood of $p$
	in $Strain_{\eps}[u]$, it is possible to take  a countable subcollection $\{(A_k,Y_k,\varphi_k)\}$ of the collection $\{(A_{p,\eps},Y_{p,\eps},\varphi_{p,\eps})\}$ such that the union of the transport sets $A_k$ covers $Strain[u]$. We may also take $A_k$ to be disjoint, by replacing $A_k$ by $A_k\setminus\bigcap_{k'<k}A_{k'}$. This does not change the property that $\varphi(Y_k)\cap\gamma$ is a singleton for every transport ray $\gamma$ in $A_k$.
\end{proof}

A \emph{parallel line-cluster} is a set $ B \subseteq \RR^2$ of the form
	\begin{equation}\label{lineclusterform}
				B = \left \{ (y,t) \in \RR^2 \, ; \, y \in Y, \ a_y < t < b_y \right \},
	\end{equation}
	where $Y \susbeteq \RR$ is a Borel set and $a: Y \rightarrow [-\infty, 0)$ and $b: Y \rightarrow (0, +\infty]$ are Borel measurable functions.

\begin{corollary}[``Local straightening of the partition''] \label{straightcor} For each of the transport sets $A_k$ from Lemma \ref{decomplemma} there is a parallel line-cluster
$$B_k =\left \{ (y,t) \in \RR^2 \, ; \, y \in Y_k, \ a_{k,y} < t < b_{k,y} \right \} $$
	and a locally Lipschitz bijection $F_k: B_k \rightarrow A_k$ such that the following hold:
	\begin{enumerate}
		\item[(a)] For every $y \in Y_k$,  $F_k(y,0) = \vphi_k(y)$.
		\item[(b)] For any $y \in Y_k$, the curve $t \mapsto F_k(y,t)$ defined for $a_y < t < b_y$ is a transport
		ray.
		\item[(c)] The map
 		\begin{equation}\label{vdef}
 		y \mapsto \frac{\partial F_k(y,t)}{\partial t} \Bigg\vert_{t=0} \end{equation}
 		is locally Lipschitz.
		\item[(d)] Write $E_k \subseteq M$ for the set of all limit points of the form $\lim_{t \rightarrow a_y^+} F_k(y,t)$ or
		$\lim_{t \rightarrow b_y^-} F_k(y,t)$, whenever the limit exists, where $y$ ranges over the set $Y_k$. Then $E_k$
		has measure zero.
	\end{enumerate}
\end{corollary}

\begin{proof}
	Fix $k \ge 1$ and let $A = A_k, Y = Y_k, \varphi = \varphi_k$ be as in Lemma \ref{decomplemma}.
	For each $y \in Y$ let $\gamma_y$ be the transport ray satisfying $\dot\gamma_y(0) = \nabla u (\varphi(y))$. Let
	$$a_y = -\alpha(\varphi(y)), \qquad b_y = \beta(\varphi(y))$$
	where $\alpha$ and $\beta$ are defined by \eqref{rayeq}. Define $B$ by \eqref{lineclusterform} and let
	$$F(y,t) = \gamma_y(t), \qquad (y,t) \in B.$$
	The transport ray $\gamma_y$ is a unit-speed geodesic, and in the case where $\lim_{t \rightarrow a_y^+} \gamma_y(t)$ exists
 we define $F(y,a_y)$ by continuity. We similarly define $F(y, b_y)$ by continuity whenever $\lim_{t \rightarrow b_y^-} \gamma_y(t)$ exists.
 	
	\medskip
	By Lemma \ref{decomplemma} the  function $F$ is a bijection on $B$, since distinct transport rays do not intersect, and the function $\nabla u \circ \varphi$ is locally Lipschitz on $Y$. Since $(\partial F / \partial t)(y,0) = \dot\gamma_y(t) = \nabla u (\vphi(y))$, property (c) follows, and moreover the function $F$ is locally Lipschitz on the set $B$ defined  in \eqref{lineclusterform}.
	The functions $a$ and $b$ are Borel since the functions $\alpha$ and $\beta$ are, so $B$ is a parallel line cluster.  The set $E_k$ is the image under the map $F$ of the set
	$$ \left\{(y,t) \in B  \ \big\vert \ t \in \{a_y, b_y\}, \ F_y(t) \text{ is defined } \right\} \subseteq \RR^{n-1} \times \RR = \RR\times \RR,$$
	which is contained in the union of two graphs of measurable functions (of one variable, as $n=2$). Since $F$ is locally Lipschitz, we conclude that $E_k$  has measure zero (see \cite[Lemma 3.1.8]{K1}).
\end{proof}

\begin{corollary}\label{straincor}
	Almost any point in $M$ belongs either to $Strain[u]$ or to $Loose[u]$. Moreover, almost all points in the support of $\mu_1 - \mu_2$ belong to $Strain[u]$.
\end{corollary}

\begin{proof}
	We claim that all points which are neither in $Strain[u]$
	nor in $Loose[u]$ are limit points of the form described in Corollary \ref{straightcor}(c). Indeed, if $y \not \in Strain[u] \cup Loose[u]$
	then we know that $(x,y) \in \Omega_u$ or $(y,x) \in \Omega_u$ for some $x \neq y$.
	In the first case, the relative interior of any minimizing geodesic from $x$ to $y$ is contained in a transport ray, and in the second case, the relative interior of any minimizing geodesic from $y$ to $x$ is contained in a transport ray. This follows from Proposition \ref{uprops}(iii). However, the point $y$ does not belong to any transport ray, since it is not in $Strain[u]$.
	Therefore $y$ must be a limit point as in (c). It now follows from Corollary \ref{straightcor}(c) that
	$Strain[u] \cup Loose[u]$ is a set of full measure in $M$. The ``moreover'' part is proven as in \cite[Theorem 1.5(B)]{K1}.
\end{proof}

\subsection{A needle decomposition theorem for Finsler surfaces}\label{finsdecompsubsec}

We can now conclude the Finslerian needle decomposition theorem.

\begin{theorem}\label{finslerdecomp}
	Let $(M,\Phi)$ be a two-dimensional geodesically-convex Finsler manifold and let $\mu$ be an absolutely continuous measure on $M$ with a smooth density. Let $\rho_1, \rho_2: M \rightarrow [0, \infty)$ be $\mu$-integrable, compactly-supported functions with
	$$ \int_M \rho_1 d \mu = \int_M \rho_2 d \mu. $$
	Then there is a collection $\Lambda$ of disjoint minimizing geodesics, a measure $\nu$ on $\Lambda$ and a family $\{\mu_\gamma\}_{\gamma \in \Lambda}$ of Borel measures on $M$ such that the following hold:
	\begin{enumerate}
		\item[(i)] For all $\gamma\in \Lambda $, the measure $\mu_{\gamma}$ is supported on $\gamma$.
		\item[(ii)] (``disintegration of measure'') For any measurable set $S \susbeteq M$,
		\begin{equation}\label{disinteqn}
		\mu(S) = \int_{\Lambda} \mu_{\gamma}(S) d \nu(\gamma).
		\end{equation}
		
		\item[(iii)] (``mass balance'') For $\nu$-almost any $\gamma \in \Lambda$
		\begin{equation}\label{MBdecomp}	\int_{M} \rho_1 d \mu_{\gamma} = \int_{M} \rho_2 d \mu_{\gamma}, \end{equation}
		and moreover
		\begin{equation}\label{MBendsdecomp}	\int_{M} \rho_1 d \mu_{\gamma^+} \le \int_{M} \rho_2 d \mu_{\gamma^+} \end{equation}
		whenever $\gamma^+$ is a positive end of $\gamma$. (Recall that a curve $\gamma^+$ is said to be a positive end of a curve $\gamma : I \to M$ if it is a restriction of $\gamma$ to a subinterval of the form $I \cap [t,\infty)$, and the measure $\mu_{\gamma^+}$ is the restriction of $\mu_\gamma$ to the image of $\gamma^+$.)
	\end{enumerate}
\end{theorem}

\begin{proof}
Let $\mu_1,\mu_2$ be the measures whose densities with respect to $\mu$ are $\rho_1,\rho_2$, respectively. We fix a solution $u$ to the optimization problem \eqref{MK}, and apply to it the results of the previous section. By Lemma \ref{decomplemma} and Corollary \ref{straightcor} there exist disjoint transport sets $(A_k)_{k \geq 1}$ whose union covers $Strain[u]$, as well as and locally Lipschitz maps $F_k : B_k \to A_k$, where $B_k$ is a parallel line-cluster:
\begin{equation}\label{lineclusterformk}
				B_k = \left \{ (y,t) \in \RR^2 \, ; \, y \in Y_k, \ a_{k,y} < t < b_{k,y} \right \},
\end{equation}
and $F_k(y,\cdot)$ is a constant speed minimizing geodesic for every $y \in Y_k$.

\medskip
Denote
	\begin{equation}
	\mathbf{A} : = \bigsqcup_kA_k.
	\end{equation}
	By Lemma \ref{decomplemma} and Corollary \ref{straincor}, the set $\mathbf{A}$ contains the support of $\mu_1 - \mu_2$ up to a set of measure zero. Since the formulation of Theorem \ref{finslerdecomp} is indifferent to altering $\rho_1,\rho_2$ on a set of measure zero we may therefore assume that
	\begin{equation}\label{rhocoincide}
	\rho_1 \equiv  \rho_2 \qquad \text{ on } M \setminus \mathbf{A}.
	\end{equation}
	We now decompose $M$ into a disjoint union of minimizing geodesics. Recall that our notation
does not distinguish between the curve $F_k(y, \cdot)$ as a map defined on an interval, and its image $\{ F_k(y, t) \, ; \, a_{k,y} < t < b_{k,y} \}$ which is just a subset of $M$.
Let
	\begin{equation*}
	\Lambda: = \bigsqcup_{k = 0}^\infty \Lambda_k, \qquad \text{where } \qquad \begin{aligned} \Lambda_k  &:= \{F_k(y , \cdot) \, ; \,  y \in Y_k\} \qquad k \ge 1,\\ \Lambda_0 &: = \{\{p\} \, ; \,   p \in M \setminus \mathbf{A} \}. \end{aligned}
	\end{equation*}
	Thus for $k \ge 1$ the set $\Lambda_k$ consists of transport rays which are images of lines $y = \mathrm{const}$ in $B_k$ under the map $F_k$, while the set $\Lambda_0$ consists of all singletons which are not contained in $\mathbf{A}$.

	\medskip
	Since the set $\mathbf{A}$ is the union of the sets $A_k = F_k(B_k)$, the disjoint union of the members of $\Lambda$ is the entire manifold $M$, and we have a map
	$$\pi: M \to \Lambda$$
	assigning to each $p \in M$ the unique member of $\Lambda$ containing  $p$.
	
	\medskip For each $k \geq 1$, the map $F_k$ is locally Lipschitz, and therefore the function
	\begin{equation} J_k = |\det dF_k| \label{eq_342} \end{equation}
	is defined almost-everywhere on $B_k$. The determinant here is understood using the Euclidean volume form on $B_k$ and using the density $\mu$ on the manifold $M$. Since $F_k$ is locally Lipschitz, it follows from Fubini's theorem and the change of variables formula (see e.g. Evans and Gariepy \cite{EG}),  that for every $\mu$-measurable function $\psi : A_k \to \RR$,
	\begin{equation}\label{changeofvariablesFk}
	\int_{A_k}\psi d\mu = \int_{B_k}(\psi \circ F_k) J_k dtdy = \int_{Y_k}\int_{a_{k,y}}^{b_{k,y}} \psi (F_k(y,t)) J_k(y,t) dtdy,
	\end{equation}
where we recall that the parallel line-cluster $B_k$ has the form \eqref{lineclusterformk}.

\medskip 	
The measures $\mu_\gamma$ are defined as follows:
	\begin{itemize}
		\item If $\gamma = F_k (y , \cdot)\in \Lambda_k$ for some $k \geq 1$ and  $y\in Y_k$, then we set
		\begin{equation}\label{mugamma}\mu_\gamma = \gamma_\#(e^{|y|}J_k(y,t)dt), \end{equation}
		i.e. $\mu_\gamma$ is the pushforward of the measure $e^{|y|}J_k(y,t)dt$ on $(a_{k,y}, b_{k,y})$ via the map $\gamma : (a_{k,y}, b_{k,y}) \to M$. The factor $e^{|y|}$ is there so that the measure $\nu_k$ defined below will be finite.
		\item If $\gamma = \{p\}\in\Lambda_0$ for some $p \in M \setminus A$, then we set
		$$\mu_\gamma = \delta_p,$$
		where $\delta$ denotes a Dirac measure.
	\end{itemize}
	Next, we define the measure $\nu$ on $\Lambda$.
	\begin{itemize}
		\item
		For each $k \ge 1$ define $f_k : Y_k \to M$ by
		$$f_k(y) = F_k(y, 0),$$
		and define a measure $\nu_k$ on $\Lambda_k$ by
		$$\nu_k : = (\pi \circ f_k)_\#(e^{-|y|}dy).$$
		Note that here $e^{-|y|}dy$ is a measure on $Y_k \subseteq \RR$. (When we push-forward a measure, we push-forward its $\sigma$-algebra as well.)
		\item
		On $\Lambda_0$ define the measure
		$$\nu_0 : = \pi_\# (\mu\vert_{M\setminus A}).$$
	\end{itemize}
	Finally, let $\nu$ be the measure on $\Lambda$ satisfying
	$$\nu\vert_{\Lambda_k} = \nu_k \qquad \text{ for all } k \geq 0.$$
	It remains to verify conclusions (ii) and (iii) of Theorem \ref{finslerdecomp}. For every measurable set $S \subseteq M$, and for every $k\ge 1$, by virtue of \eqref{changeofvariablesFk},
	\begin{align*}
	\mu(S \cap A_k) & = \int_{Y_k} \int_{a_{k,y}}^{b_{k,y}}\chi_{S}(F_k(y,t)) J_k(y,t) dt dy \\
	& = \int_{Y_k}\int_{\pi(f_k(y))}\chi_S(x)e^{-|y|}d\mu_{\pi(f_k(y))}(x) dy\\
	& = \int_{Y_k}  \mu_{\pi(f_k(y))}(S) e^{-|y|} dy\\
	& = \int_{\Lambda_k}\mu_\gamma(S)d\nu_k(\gamma).
	\end{align*}
	Summing this over $k$ gives
	\begin{equation}\label{muSA}
	\mu(S \cap \mathbf{A}) = \sum_{k=1}^\infty\int_{\Lambda_k}\mu_\gamma(S) d\nu_k(\gamma).
	\end{equation}
	Moreover, by our definition of $\nu_0$,
	\begin{align*}
	\mu(S \setminus \mathbf{A})
	& = \int_{M \setminus \mathbf{A}} \chi_{S}(p) d\mu(p)\\
	& = \int_{M \setminus \mathbf{A}} \delta_p(S) d\mu(p)\\
	& = \int_{\Lambda_0} \mu_{\gamma}(S) d\nu_0(\gamma)
	\end{align*}
	which together with \eqref{muSA} and the definition of $\nu$ implies \eqref{disinteqn}.
	In order to prove \eqref{MBdecomp}, it suffices to prove that for every $\nu$-measurable subset $T \subseteq \Lambda$,
	\begin{align}\label{Sinteq}
	\int_T \left(\int_M\rho_1 d\mu_\gamma\right) d\nu(\gamma) &= \int_T \left(\int_M\rho_2 d\mu_\gamma\right) d\nu(\gamma).
	\end{align}
	To this end, we observe that for $i = 1,2$,
	\begin{align} \nonumber
	\int_T \left(\int_M\rho_i d\mu_\gamma\right) d\nu(\gamma) & = \int_{\Lambda} \left(\int_M \chi_{\pi^{-1}(T)}  \rho_i d\mu_\gamma\right) d\nu(\gamma)\\
	& = \int_{\pi^{-1}(T)} \rho_i d\mu\\
	& = \sum_{k=0}^\infty\int_{\pi^{-1}(T\cap\Lambda_k)}\rho_id\mu,
	\label{intrhoi}
	\end{align}
	where the second equality follows from \eqref{disinteqn}. Now observe that if $k \ge 1$ then $\pi^{-1}(T\cap\Lambda_k)$ is a transport set so
	\begin{equation}\label{kwiseequality} \int_{\pi^{-1}(T\cap\Lambda_k)}\rho_1d\mu = \int_{\pi^{-1}(T\cap\Lambda_k)}\rho_2 d\mu\end{equation}
	by Proposition \ref{uprops}(iv), while if $k=0$ then \eqref{kwiseequality} follows from \eqref{rhocoincide} since $\pi^{-1}(T\cap\Lambda_0) \subseteq M \setminus\mathbf{A}$. Thus \eqref{Sinteq} is proven.

	\medskip
	Finally, let $\hat\Lambda \subseteq \Lambda$ be a $\nu$-measurable set such that for every $\gamma \in\hat\Lambda$, inequality \eqref{MBendsdecomp} fails for some positive end $\gamma^+$ of $\gamma$. We may choose such a positive end $\gamma^+$ for each $\gamma \in \hat \Lambda$ so that the set $S^+:=\bigcup_{\gamma \in \hat \Lambda}\gamma^+$ is a Borel set and therefore a positive end of a transport set. It then follows from Proposition \ref{uprops}(iv) that
	$$0\le \int_{S^+}(\rho_2-\rho_1)d\mu = \int_{\hat\Lambda}\left(\int(\rho_2-\rho_1)d\mu_{\gamma^+}\right) d\nu(\gamma).$$
	Since the integrand is negative by the choice of $\gamma^+$, we must conclude that $\nu(\hat\Lambda) = 0$. This finishes the proof of Theorem \ref{finslerdecomp}.
\end{proof}

\subsection{Proof of Theorem \ref{horodecomp}}
The proof of Theorem \ref{horodecomp} is based on the observation that the Finslerian needle decomposition Theorem \ref{finslerdecomp} is oblivious to the  parametrization of the geodesics (though their orientation does matter). We use the fact that oriented horocycles coincide, as directed curves, with the geodesics of the metric \eqref{horofinsler}, in order to apply Theorem \ref{finslerdecomp} and obtain conclusions (i)-(iii) of Theorem \ref{horodecomp}. The fact that the density of the needles is affine, which is the content of Theorem \ref{horodecomp}(iv), is a consequence of the following key lemma about horocycles. Roughly speaking, it asserts that a coordinate system $(y,t)$ of the hyperbolic plane in which the curves $y = \mathrm{const}$ are horocycles and $t$ is arclength, has an area distortion which is affine in $t$.

\begin{lemma}\label{detdFaffine}
	Let $Y \subseteq \RR$ be a Borel set and let
		$$\lambda : Y \to (0,\infty), \qquad t_0 : Y \to \RR \qquad \text{ and } \qquad \omega: Y \to S^1$$
	be locally Lipschitz functions. Define
		\begin{equation}\label{Fformula} F(y,t) : = \alpha_{\lambda(y),t_0(y),\omega(y)}(t), \qquad y \in Y, \qquad t \in \RR, \end{equation}
	 where $\alpha$ is defined in \eqref{alphaformula}. Then for almost every $y \in Y$, the function
		$$t \mapsto \det dF(y,t)$$
	is affine-linear.
\end{lemma}

\begin{proof}

Pick a point $y \in Y$ where the maps $\lambda,t_0,\omega$ are differentiable; this holds for almost every $y \in Y$ by the Rademacher theorem
on the differentiability of locally Lipschitz functions. Then for all $a_y < t < b_y$, we see from (\ref{Fformula}) that the map $F$ is differentiable at $(y,t)$. The determinant of $dF$ with respect to the \emph{Euclidean} area forms on the domain $Y \times \RR$ and the range $\CC$, is given by
	\begin{equation}\label{dFeuc}
	 \mathrm{Im}\left(\overline{\partial_t F(y,t)}\partial_y F(y,t) \right).
	\end{equation}
	Write $\omega = e ^ {i\phi}$. Then by \eqref{Fformula},
	\begin{equation*}\label{dF}
		\partial_tF = \frac{2\lambda ie^{i\phi}}{\left((t-t_0) + i (1 + \lambda)\right)^2},
		\qquad
		 \partial_yF = \frac{ie^{i\phi}\Big( \phi' \left( (t - t_0 + i)^2 + \lambda^2 \right) - 2 ( t - t_0 + i) \lambda' -2 t_0'\lambda\Big) }{\left((t-t_0) + i (1 + \lambda)\right)^2}
	\end{equation*}
	and
	\begin{equation*}
		|F|^2 = \frac{(t-t_0)^2+(1-\lambda)^2}{(t-t_0)^2 + (1+\lambda)^2} = 1 - \frac{4 \lambda}{(t-t_0)^2 + (1+\lambda)^2},
	\end{equation*}
	whence
	\begin{align*}
	\mathrm{Im}\left(\overline{\partial_t F}\partial_y F\right)
	& = \frac{4 \lambda \phi' (t-t_0) - 4 \lambda \lambda'}{\left((t-t_0)^2 + (1+\lambda)^2\right)^2}\\
	& = \frac{(1 - |F|^2)^2}{4 \lambda } \cdot \big( (t - t_0 ) \phi '  - \lambda ' \big).
	\end{align*}
	In the Poincar\'e disc model, the hyperbolic area density form is given by $4|dz|^2 / (1 - |z|^2 )^2$. It follows that, this time with respect to the Euclidean area form on $Y \times \RR$ and the \emph{hyperbolic} area form on the unit disc,
	\begin{equation}\label{Jformula}
		\det dF =  \lambda ^ {-1} ((t - t_0)\phi ' - \lambda ')
	\end{equation}
	at any point $(y,t)$ such that $\lambda, t_0, \phi$ are differentiable at $y$. Since the expression on the right hand side of \eqref{Jformula} is affine-linear in $t$, we are done.
\end{proof}

\begin{proof}[Proof of Theorem \ref{horodecomp}]
	 Let $M = \HH^2$ be the hyperbolic plane, which we identify here with the unit disc equipped with the Poincar{\'e} metric (\ref{eq_1109}). Let $\mu$ be the hyperbolic area form and let $\rho_1,\rho_2$ satisfy the hypotheses of Theorem \ref{horodecomp}. Apply Theorem \ref{finslerdecomp} to the Finsler manifold $(M,\Phi)$ where $\Phi$ is given by \eqref{horofinsler}, and with $\mu$, $\rho_1$ and $\rho_2$ as above. By Subsection \ref{horosubsec}, the Finsler manifold $(M,\Phi)$ is geodesically convex and its geodesics are oriented horocycles, though not parametrized by hyperbolic arclength. Conclusions (i), (ii) and  (iii) of Theorem \ref{horodecomp} thus follow immediately from Theorem \ref{finslerdecomp}. It remains to prove conclusion (iv):
	\begin{itemize}
	\item[($\ast$)] For $\nu$-almost every $\gamma \in \Lambda$, the density of $\mu_\gamma$ with respect to arclength along $\gamma$ is affine-linear, except when $\gamma$ is a singleton, in which case the measure $\mu_\gamma$ is a Dirac mass.
	\end{itemize}
	We work under the notations introduced in the proof of Theorem \ref{finslerdecomp}. When $\mu_\gamma$ is not a Dirac mass, it is given by \eqref{mugamma}:
		$$\mu_\gamma = \gamma_\#(e^{|y|}J_k(y,t)dt),$$
	where $J_k(y,t) = |\det dF_k(y,t)|$ and $F_k$ is one of the functions from Corollary \ref{straightcor}. Write $F = F_k$.
	Recall that $F$ is defined on a Borel set $B\subseteq \RR^2$ of the form
	$$ B = \left \{ (y,t) \, ; \, y \in Y, \, a_y < t < b_y \right \}$$
	and that $F(t,\cdot)$ is a geodesic of $\Phi$ for almost every $y \in Y$.
	
	\medskip
	Since the geodesics of $\Phi$ correspond to oriented horocycles up to orientation-preserving reparametrization, there exists a function $t = t(y,\tau)$, stricty increasing in $\tau \in (\tilde{a}_y, \tilde{b}_y)$ and onto the interval $(a_y, b_y)$, such that the following holds: if we set
	\begin{equation}
		\tilde F(y,\tau) = F(y,t(y,\tau))
	\end{equation}
	then the curve
		$$\tilde \gamma (\tau) = \tilde F(y,\tau)$$
	is a \emph{unit-speed} oriented horocycle for almost every $y \in Y$.  The function $t$ is locally Lipschitz by Corollary \ref{straightcor}(c), together with the fact that both the hyperbolic metric and the metric $\Phi$ are positive and smooth away from the zero section. By the chain rule we then have
		\begin{equation}\label{mugammatau}\mu_\gamma = \tilde\gamma_\#(e^{|y|}\tilde J(y,\tau)d\tau)\end{equation}
	where $\tilde J := |\det d \tilde F|$.

	\medskip
	Since $\tilde F(y,\cdot)$ is a unit-speed horocycle we may write, in the disc model,
	\begin{align}
	\tilde F(y,\tau) &=  \alpha_{\lambda(y),t_0(y),\omega(y)}(\tau) \qquad \qquad (\tilde{a}_y < \tau < \tilde{b}_y)
	\end{align}
	for some functions $\lambda, t_0, \omega: Y \rightarrow (0,\infty) \times \RR \times S^1$, where $\alpha$ is defined in \eqref{alphaformula}. The map
	\begin{equation} y\mapsto(\lambda(y), t_0(y), \omega(y)) \label{eq_1026} \end{equation} is locally Lipschitz. Indeed, it is the composition of the map $y \mapsto (\partial \tilde F / \partial \tau )(y,0)$ with the map sending a unit tangent vector to the triple $(\lambda, t_0, \omega)$ corresponding to the horocycle passing through that vector at time 0. The first map is locally Lipschitz by Corollary \ref{straightcor}(c), and the second is smooth by the formulae \eqref{parametereq}. We can therefore use Lemma \ref{detdFaffine} and conclude that the map $\tau \mapsto \det d\tilde F (y,\tau)$ is affine-linear.

	\medskip
	Since $\tilde J = |\det d\tilde F|$, and in view of formula \eqref{mugammatau}, in order to finish the proof it remains to prove that for almost every $y\in Y$,
the expression $\det d\tilde F(y,\tau)$ does not change sign. Equivalently, we must show that $\det dF(y,t)$ does not change at any $t \in (a_y,b_y)$, which is precisely the content of Lemma \ref{lem_sign} below.
\end{proof}

\begin{lemma} Let $F_k : B_k \to A_k$ be one of the maps from Corollary \ref{straightcor}. Assume that $M$ is orientable. Then for almost any $y_0 \in Y_k$, the
function $t \mapsto \det d F_k(y_0,t)$ does not change sign in the interval $(a_{y_0}, b_{y_0})$.
\label{lem_sign}
\end{lemma}

\begin{proof} We may assume that $y_0$ is a Lebesgue density point of $Y := Y_k$
and is a Lebesgue point of the functions $a_y$ and $b_y$. Thus
$ \liminf_{Y \ni y \rightarrow y_0} a_y \leq a_{y_0}$ while $ \limsup_{Y \ni y \rightarrow y_0} b_y \geq b_{y_0}$.
Assume by contradiction that there exist $t_1 < t_2$ in the interval $(a_{y_0}, b_{y_0})$ such that
\begin{equation}\label{detdFsigns} \det dF(y_0,t_1) \cdot \det dF(y_0,t_2) < 0,  \end{equation}
i.e., the linear basis
$$ \frac{\partial F}{\partial t}(y_0, t), \frac{\partial F}{\partial y}(y_0, t) \in T_{F(y_0, t)} M $$
has a different orientation for $t = t_1$ and for $t = t_2$. The curve $t \mapsto F(y_0, t)$ is a simple regular curve
defined for $t \in (a_{y_0}, b_{y_0})$, and the vector $\frac{\partial F}{\partial t}(y_0, t)$ is a smooth, non-zero tangent to this curve.
Hence, by \eqref{detdFsigns}, the two vectors $(\partial F/\partial y)(y_0,t_1)$ and $(\partial F/\partial y)(y_0,t_2)$ point at different sides of the curve $\gamma$. This implies that the curves $F(y,\cdot)$ and $F(y_0,\cdot)$ must cross each other for some $y$ sufficiently close to $y_0$, but $F$ is a bijection by Corollary \ref{straightcor}, a contradiction.
\end{proof}

\section{Appendix: proof of Lemma \ref{tanglemma}}

Let $(M,\Phi)$ be a two-dimensional Finsler manifold. Let $\pi : TM \to M$ denote the tangent bundle map of $M$, and let $SM\subseteq TM$ denote the unit tangent bundle. For each $v \in SM$, let $\gamma_v$ denote the unit-speed geodesic satisfying $\dot\gamma_v(0) = v$, defined on its maximal domain of definition.
\medskip As in \cite[Section 2.1]{BCS} at each $p \in M$ we consider the {\it Hessian metric} induced by $\Phi^2/2$ on the fiber $T_p M$.
That is, we consider the strongly convex function $\Phi^2/2$ in the linear
space $T_p M$, and for $u,v,w \in T_p M$ with $u \neq 0$ we set
$$g_u(v,w) =  \frac{1}{2} \partial_v \partial_w  \Phi^2(u), $$
where $\partial_v$ is directional derivative in the direction of the vector $v$ in the linear space $T_p M$.
Since $\Phi$ is $1$-homogeneous in $T_p M$, from the Euler identity we know that
$$ g_{u}(u,v) = \partial_v (\Phi^2/2). $$
 If $U\subseteq M$ is an orientable domain, then there exists a smooth map $J: SU \to SU$ (here $SU$ is the unit tangent bundle of $(U,\Phi)$) satisfying
\begin{equation}\label{Jortho} g_{J(v)}(J(v),v) = 0, \qquad v \in SU.\end{equation}
Indeed, by orientability we can choose smoothly for each $v \in SU$ a unit covector $v_\perp \in S^* U$ in the annihilator of $v$.
Then we take $J(v) = \cL(v_\perp)$, where $\cL$ is the Legendre transform (see Section \ref{backgroundsec}). Now (\ref{Jortho}) follows by a standard argument in convex analysis.

\medskip
The following lemma asserts the existence of ``parallel geodesic coordinates" in any direction in $SM$. Every time we work in local coordinates, we will, as we may, assume that for every $p,q$ in the coordinate neighbourhood, the quantities $d(p,q), d(q,p)$ and $|p-q|$ (where the latter stands for Euclidean distance in coordinates), are all comparable up to some multiplicative factor depending only on the neighbourhood. For a subset $A \subseteq M$ and a point $x \in M$ we write
$d(x,A) = \inf_{y \in A} d(x,y)$ and
$d(A,x) = \inf_{y \in A} d(y,x)$.

\begin{lemma}(``Existence of geodesic parallel coordinates'')\label{parcor}
	Let $(M,\Phi)$ be a Finsler surface. Then for every $p \in M$ there exist $s_0,t_0 > 0$ and a neighbourhood $U \ni p $ with the following property: for every $v \in SU$ there exists a coordinate chart $\psi_v: (-s_0,s_0) \times (-t_0, t_0) \to U_v\subseteq M$, whose image $U_v$ contains $U$ and such that
	$$\psi_v(s,0) = \gamma_v(s), \qquad s \in (-s_0,s_0),$$
	and
		\begin{align}
			d_M(\psi_v(s,t), \gamma_v) &= d_M(\psi_v(s,t), \gamma_v(s)) = -t, & s \in (-s_0, s_0), & t \in (-t_0,0), \label{psidist1} \\
			d_M(\gamma_v,\psi_v(s,t)) &= d_M(\gamma_v(s),\psi_v(s,t)) = t, & s \in (-s_0, s_0), & t \in (0,t_0). \label{psidist2}
		\end{align}
	In particular, for all $s \in (-s_0,s_0)$, the curve $\psi_v(s,\cdot)$ is a forward-minimizing geodesic. Moreover, the transition maps $\{\psi_u^{-1}\circ\psi_v\}_{u,v \in SU}$ all have $C^3$ norms bounded by a constant indepenent of $u,v$.
\end{lemma}

\begin{proof}
	Let $U$ be a strongly convex neighbourhood of $p$, and let $J$ be the map discussed above. Possibly after shrinking  $U$, there exists a rectangle $Q = (-s_0,s_0) \times (-t_0, t_0) \subseteq \RR^2$ such that for every $v \in SU$, the map
	$$\psi_v : Q \to M, \qquad \psi_v(s,t) = \pi\circ\phi_t\circ J\circ\phi_s(v),$$
	where $\phi$ denotes the geodesic flow on $TM$, is defined and smooth on all of $Q$. Note that
	for all $s \in (-s_0, s_0)$,
	 $$ \psi_v(s,0) = \gamma_v(s). $$
	The differential of $\psi_v$ is nonsingular at $(0,0)$. In fact,
	$$d\psi_v(0,0) = v \otimes ds + Jv \otimes dt.$$
	By the inverse function theorem, possibly after shrinking
	$U$ and decreasing $s_0$ and $t_0$, 	
	we may assume that the map $\psi_v$ is bijective from $Q$ to $U_v : = \varphi_v(Q)$, for every $v \in TU$.
	Moreover, we may assume that  both $\psi_v$ and $\psi_v^{-1}$ have $C^3$ norm bounded uniformly over all $v \in TU$, where the $C^3$ norm is taken with respect to some fixed coordinate chart on $U$. This also implies that the image $U_v$ contains a ball of radius $r_0 = r_0(p, t_0, s_0)$ in this coordinate chart; thus, by taking $U$ smaller, we may assume that $U\subseteq U_v$ for every $v \in TU$.	
	
	\medskip We prove the second equality \eqref{psidist2}. First, by the triangle inequality, if we take $t_0$ small enough, then for any $s \in (-s_0/2,s_0/2)$ and $t\in (0,t_0)$, the intersection of $\gamma_v$ with any backward ball of radius $t_0$ centered at $\psi_v(s,t)$ is contained in $\gamma_v\vert_{(-s_0,s_0)}$. Therefore, as the distance from $\gamma_v$ to $\psi_v(s,t)$ is at most $t_0$, it must be attained at a point $\gamma_v(\hat s)$, for some $\hat s \in (-s_0,s_0)$.
	
	\medskip By \eqref{Jortho} and the first variation formula (see \cite[Section 5.1]{BCS}), the geodesic from $\gamma_v(\hat s)$ to $\psi_v(s,t)$ has initial velocity $J(\dot\gamma_v(\hat s))$ (here we use the fact that $t_0 > 0$, because then $J(\dot\gamma_v(\hat s))$ points at the side of $\gamma_v$ where $\psi_v(s,t)$ lies).
	
	\medskip By the definition of $\psi_v$ we have $\psi_v(\hat s,\hat t) = \psi_v(s,t)$ for some $\hat t \in (0,t_0)$, but since $\psi_v$ is bijective we conclude that $\hat s = s$ and $\hat t = t$, i.e., the closest point to $\psi_v(s,t)$ is $\gamma_v(s)$ and the distance is $t$. This proves the second equation, after replacing $s_0$ by $s_0/2$. The proof of \eqref{psidist1} is similar.
\end{proof}

Let us recall the formulation of Lemma \ref{tanglemma}:

\begin{lemma}[``Disjoint geodesics that are close to each other at one point have similar tangents'']\label{tanglemma_} Let $p \in M$. Then there exist $C,c > 0, \sigma_0 \in (0,1)$ and a coordinate chart $\psi: V \to U$, where $V$ is an open subset of $\RR^2$ and $U$ is a neighbourhood of $p$, such that the following holds.
	
	\medskip  Suppose that
	$0 < \sigma \leq \sigma_0$ and $\gamma_1, \gamma_2: (-\sigma, \sigma) \rightarrow U$ are two disjoint, unit-speed, forward geodesics. Set
	\begin{equation}\label{etadef}\eta_i = \psi^{-1} \circ \gamma_i,\end{equation}
and assume that
	$|\dot{\eta}_1(0) - \dot{\eta}_2(0)| < c$. Then,
	\begin{equation}\label{etadotdist} |\dot{\eta}_1(0) - \dot{\eta}_2(0)| \leq \frac{C}{\sigma} \cdot |\eta_1(0) - \eta_2(0)|. \end{equation}
\end{lemma}

For the proof of Lemma \ref{tanglemma_} we require the following little lemma:

\begin{lemma}
	Let $0 < \delta \le 1, A > 0$. Let $y: (-\delta,\delta) \to \RR$ be a positive smooth function satisfying $|y''| \le A (|y| + |y'|)$. Then $|y'(0)| \le B y(0) / \delta$, where $B$ is a constant depending on $A$ alone. \label{lem_1006}
\end{lemma}

\begin{proof}[Proof of Lemma \ref{tanglemma_}]
	There is no loss of generality in assuming that $y(0) = 1$ and $y'(0) \le 0$. Since $y(\delta/2)$ is positive, there exists $x \in (0,\delta/2)$ where $y'(x) \ge - 2/\delta$. Let $x_0$ be the minimal such point. Then for all $x\in[0,x_0]$ we have $y'(x) \le -2/\delta$ and $0 < y(x) \le 1$. In particular, $|y(x)| \le |y'(x)|$, and by our assumption this implies that $y''(x) \le 2A |y'(x)| = - 2A y'(x)$ for $x \in [0,x_0]$. It follows  that $e^{2Ax}y'(x)$ is decreasing for $x \in [0,x_0]$. Thus
	\begin{equation*}
	0 \ge y'(0) \ge y'(x_0)e^{2Ax_0} \ge -2e^{2A}/\delta
	\end{equation*}
	as desired.
\end{proof}

\begin{proof}
	Let $p \in M$, and let $U$ be the neighbourhood of $p$ supplied by Lemma \ref{parcor}. We need to provide a  coordinate chart for which the conclusion of the lemma holds true, and we arbitrarily choose any of the $\psi_u$
	supplied by Lemma \ref{parcor}. Note that $U$ is contained in the image of $\psi_u$.  We use the notation $A \lesssim B$ to indicate $A \leq C B$ where $C > 0$ is a constant depending only on the manifold $M$ and the point $p$.
	We will set  $\sigma_0 = \min \{s_0, t_0, r_0, 1 \}$  for $s_0$ and $t_0$ from Lemma \ref{parcor}, and where $r_0 > 0$ will be described below, and will be independent of the choice of the geodesics $\gamma_1$ and $\gamma_2$.

		\medskip Let $\gamma_1, \gamma_2$ be two  disjoint geodesics contained in $U$, defined on some symmetric interval $(-\sigma, \sigma) \subseteq \RR$. We need to find $C, c > 0$, independent of $\gamma_1$ and $\gamma_2$, such that
	(\ref{etadotdist}) holds true assuming $|\dot{\eta}_1(0) - \dot{\eta}_2(0)| < c$ and assuming $\sigma \leq \sigma_0$, where $\eta_i$ is defined in (\ref{etadef}).

	\medskip Recall that the transition maps $\psi_u^{-1}\circ\psi_v$ from the conclusion of Lemma \ref{parcor} all have $C^3$ norms uniformly bounded by a constant. If we define $\eta_i$
	via (\ref{etadef}) with $\psi = \psi_u$
	or with $\psi = \psi_v$, we may obtain different curves in $\RR^2$.
	However, if we switch from $\psi_u$ to $\psi_v$
	then both expressions
	$$ |\eta_1(0) - \eta_2(0)| \qquad \text{and}
	\qquad |\dot{\eta}_1(0) - \dot{\eta}_2(0)| $$
	can change at most by a multiplicative constant independent of $u$ and $v$. We may therefore
	switch to {\it any} coordinate chart $\psi_v$ for any $v \in SU$, and find suitable constants $C, c$ with respect to the new chart. We may even choose a chart that depends on
	the geodesics $\gamma_1$ and $\gamma_2$.

	\medskip  We choose to work with the chart $\psi := \psi_{\dot\gamma_1(0)}$, and define $\eta_i$ as in \eqref{etadef}. Then
	$$\eta_1(t) = (t,0) \qquad t \in (-\sigma, \sigma). $$
	 Thus, if we write
	$$\eta_2(t) = (x(t),y(t)), \qquad t \in (-\sigma,\sigma),$$
	then the inequality \eqref{etadotdist} that we need to prove reads
	\begin{equation}\label{etadotdistcoor}| (1,0) - (x'(0),y'(0))| \leq C \cdot \sigma^{-1}|(x(0),y(0))|, \end{equation}
	assuming that $|(1,0) - (x'(0),y'(0))|$ is less than some positive constant and that $\sigma \leq \sigma_0$. To prove this, we first claim that if $|(1,0) - (x'(0),y'(0))|$ is less than some constant then \begin{equation}\label{doteta2minus10}|(1,0) - (x'(t),y'(t))| \lesssim |y(t)| + |y'(t)|, \qquad t \in (-\sigma,\sigma).\end{equation}
	Let us prove this claim. Fix $t \in (-\sigma,\sigma)$, and let $\Phi_0 = \Phi\vert_{(x(t),0)}$. That is, $\Phi_0$ is the Finsler norm $\Phi$, evaluated at $(x(t),0)$ and viewed as a norm on $\RR^2$ via the identification $TU \cong TV \cong V \times \RR^2$. Note that since $\eta_2$ is unit-speed, and since $\Phi$ is locally Lipschitz,
	\begin{equation}\label{phi0ineq1}|\Phi_0(x'(t),y'(t)) - 1| = \Big|\Phi\vert_{(x(t),0)}(x'(t),y'(t)) - \Phi\vert_{(x(t),y(t))}(x'(t),y'(t))\Big| \lesssim |y(t)|.\end{equation}
 On a certain neighbourhood of the point $(1,0)$ in $\RR^2$, the map $(v^1,v^2) \mapsto (\Phi_0(v^1,v^2),v^2)$ is bi-Lipschitz. If $|(x'(0),y'(0)) - (1,0)|$ is less than some constant, then the point $(x'(0),y'(0))$ lies in this neighbourhood, and if $\sigma \leq \sigma_0$ and $\sigma_0$ is smaller than some constant $r_0 > 0$, then $(x'(t),y'(t))$ lies in this neighbourhood for all $t \in (-\sigma,\sigma)$. Therefore
	\begin{equation}\label{phi0ineq2}|(x'(t),y'(t)) - (1 , 0)| \lesssim |\Phi_0(x'(t),y'(t)) - \Phi_0(1,0)| + |y'(t)| = |\Phi_0(x'(t),y'(t)) - 1| + |y'(t)|,\end{equation}
	and \eqref{doteta2minus10} follows from \eqref{phi0ineq1} and \eqref{phi0ineq2}. To prove \eqref{etadotdistcoor} and complete the proof of the lemma, all that remains is to show that
	\begin{equation}\label{ydotfinal}|y'(0)| \lesssim |y(0)|/\sigma.\end{equation}
	The curves $\eta_1 (t)= (t,0)$ and $\eta_2(t) = (x(t),y(t))$ are unit-speed geodesics, so by Lipschitz continuity of the Christoffel symbols in our coordinate chart,
	\begin{equation} |y''|\lesssim |y| + |(x',y') - (1,0)| \lesssim |y| + |y'|, \label{eq_1527}
	\end{equation}
	where the second inequality follows from \eqref{doteta2minus10}.
		Now, if $|\eta_1(0) - \eta_2(0)| \ge c'\sigma$, then we get \eqref{etadotdistcoor} immediately since the left hand side is bounded by a constant. Thus, we may assume that \begin{equation}\label{gammasdist} |\eta_1(0) - \eta_2(0)| \le c \sigma.\end{equation} If this $c$ is taken small enough, then  \eqref{gammasdist} implies that $y(t)$ does not change sign when $t \in (-\sigma/2, \sigma/2)$, because otherwise $\eta_2$ would intersect the $x$ axis within a distance of less than $\sigma$ from the origin, which would imply that $\eta_1$ and $\eta_2$ intersect. Assume without loss of generality  that $y(t)$ is positive for all $t \in (-\sigma/2, \sigma/2)$. Then \eqref{ydotfinal} follows from (\ref{eq_1527}) and Lemma \ref{lem_1006}.
\end{proof}

\bigskip

\noindent
Department of Mathematics, Weizmann Institute of Science, Rehovot 76100, Israel. \\
{\it e-mails:} \verb"{rotem.assouline,boaz.klartag}@weizmann.ac.il"

\end{document}